\documentclass[hidelinks,11pt]{article}
\usepackage{textcomp}
\usepackage{fullpage}
\usepackage{hyperref}       
\usepackage{amsfonts}       
\usepackage{amsmath, amssymb, graphicx, url} 
 
\usepackage{amsthm} 
\usepackage{nicefrac}     
\usepackage{color}
\usepackage{algorithm}
\usepackage{algpseudocode}
\usepackage{mathtools}
\usepackage{epstopdf}
\usepackage{bm}
\usepackage{bbm}
\usepackage{tabulary}
\usepackage[mathscr]{euscript}
\usepackage{mathrsfs}
\usepackage{cite}
\usepackage{float}
\allowdisplaybreaks

\newtheorem{theorem}{\bf Theorem}
\newtheorem{lemma}{\bf Lemma}
\newtheorem{proposition}{\bf Proposition}
\newtheorem{corollary}{\bf Corollary}

\newtheorem{assumption}{\bf Assumption}

\theoremstyle{remark}
\newtheorem{remark}{\bf Remark}

\newcommand{\paren}[1]{\ensuremath{\left( #1\right)}}

\newcommand{\clint}[1]{\ensuremath{\left[ #1\right]}}
\newcommand{\set}[1]{\ensuremath{\left\{ #1\right\}}}
\newcommand{\matr}[1]{\ensuremath{\clint{\begin{array} #1 \end{array}}}}
\newcommand{\norm}[1]{\ensuremath{\left\| #1\right\|}}
\newcommand{\snorm}[1]{\ensuremath{\| #1\|}}

\newcommand{\R}{\ensuremath{\mathbb{R}}}

\newcommand{\F}{\ensuremath{\mathcal{F}}}

\newcommand{\E}{\ensuremath{\mathbb{E}}}

\newcommand{\C}{\ensuremath{\mathcal{C}}}

\newcommand{\T}{\ensuremath{\mathcal{T}}}

\newcommand{\LL}{\ensuremath{\mathcal{L}}}
\DeclareFontFamily{OT1}{pzc}{}
\DeclareFontShape{OT1}{pzc}{m}{it}{<-> s * [1.200] pzcmi7t}{}
\DeclareMathAlphabet{\mathpzc}{OT1}{pzc}{m}{it}
\newcommand{\LAG}{\ensuremath{{\mathpzc{L}_{\hspace{.1pt}}}}}

\newcommand{\Neg}[1]{\hspace{- #1 bp}}
\DeclareMathOperator{\Tr}{\mathrm{tr}}

\graphicspath{{./figures/}}
\title{Linear Quadratic Control with Risk Constraints	\thanks{This work was supported by the AFOSR under grant FA9550-19-1-0265 (Assured Autonomy in Contested Environments) and by the ARL under grant DCIST CRA W911NF-17-2-0181.}}
\author{Anastasios~Tsiamis	\thanks{Department   of   Electrical   and   Systems  Engineering,  University  of  Pennsylvania,  Philadelphia,  PA  19104 (email: \{atsiamis, aribeiro, pappasg\}@seas.upenn.edu). },~Dionysios~S.~Kalogerias	\thanks{Department of Electrical Engineering, Yale University, New Haven, CT 06511 (email: dionysis.kalogerias@yale.edu).}, Alejandro~Ribeiro\footnotemark[2], and George~J.~Pappas\footnotemark[2] %
	}
\usepackage{times}

\begin{document}
\maketitle
\begin{abstract}
We propose a new risk-constrained formulation of the classical Linear Quadratic (LQ) stochastic control problem for general partially-observed systems. Our framework is motivated by the fact that the risk-neutral LQ controllers, although optimal in expectation, might be ineffective under relatively infrequent, yet statistically significant extreme events.
 To effectively trade between average and extreme event performance, we introduce a new risk constraint, which explicitly restricts the total expected predictive variance of the state penalty by a user-prescribed level.
We show that, under certain conditions on the process noise, the optimal risk-aware controller can be evaluated explicitly and in closed form. In fact, it is affine relative to the minimum mean square error (mmse) state estimate. The affine term pushes the state away from directions where the noise exhibits heavy tails, by exploiting the third-order moment~(skewness) of the noise. The linear term regulates the state more strictly in riskier directions, where both the prediction error (conditional) covariance and the state penalty are simultaneously large; this is achieved by inflating the state penalty within a new filtered Riccati difference equation. We also prove that the new risk-aware controller is internally stable, regardless of parameter tuning, in the special cases of i) fully-observed systems, and ii) partially-observed systems with Gaussian noise.
The properties of the proposed risk-aware LQ framework are lastly illustrated via indicative numerical examples. 
\end{abstract}

\section{Introduction}\label{Section_Introduction}
In the problem of Linear Quadratic (LQ) stochastic control, one is typically interested in optimizing average control performance for linear systems of the form
\begin{equation}\label{FOR_EQN_system}
\begin{aligned}
x_{t+1}&=Ax_{t}+Bu_{t}+w_{t+1}\\
y_t&=Cx_t+v_t,
\end{aligned}
\end{equation}
where $x_t\in\R^{n}$ is the state, $y_t\in\R^{m}$ is the measured output, $u_t$ is input, and $w_t$, $v_t$ are process and measurement noise disturbances. A standard approach is to minimize the expectation of the following quadratic cost comprising of stage-wise input and state penalties up to a horizon $N$
 \begin{equation}\label{eq:Risk_Neutral_LQ}
\begin{aligned}
\min_{u}&\quad \E\set{ x'_NQ x_N+\sum_{t=0}^{N-1} x'_tQx_t+u'_tRu_t}\\
\mathrm{s.t.}&\quad\text{Dynamics~\eqref{FOR_EQN_system}}\\
\end{aligned},
\end{equation}
where matrices $Q,\,R$ are design choices.

While LQ control has been a standard approach to controlling stochastic systems, it only focuses on \emph{average} performance, which might be an insufficient objective when dealing with critical applications. Examples of such applications
appear naturally in many areas, including wireless industrial control \cite{Ahlen2019}, energy \cite{Bruno2016,Moazeni2015}, finance \cite{Markowitz1952,Follmer2002,Shang2018}, robotics \cite{Kim2019,Pereira2013}, networking \cite{Ma2018}, and safety \cite{samuelson2018safety,chapman2019cvar}, to name a few. Indeed, occurrence of less probable, non-typical or unexpected events might lead the underlying dynamical system to experience shocks with possibly catastrophic consequences, e.g., a drone diverging too much from a given trajectory in a hostile environment, or an autonomous vehicle crashing onto a wall or hitting a pedestrian. In such situations, design of effective \emph{risk-aware} control policies is highly desirable, systematically compensating for those extreme events, at the cost of slightly sacrificing average performance under nominal conditions.

\begin{figure}[t]
	\centering
	\includegraphics[width=0.8\columnwidth]{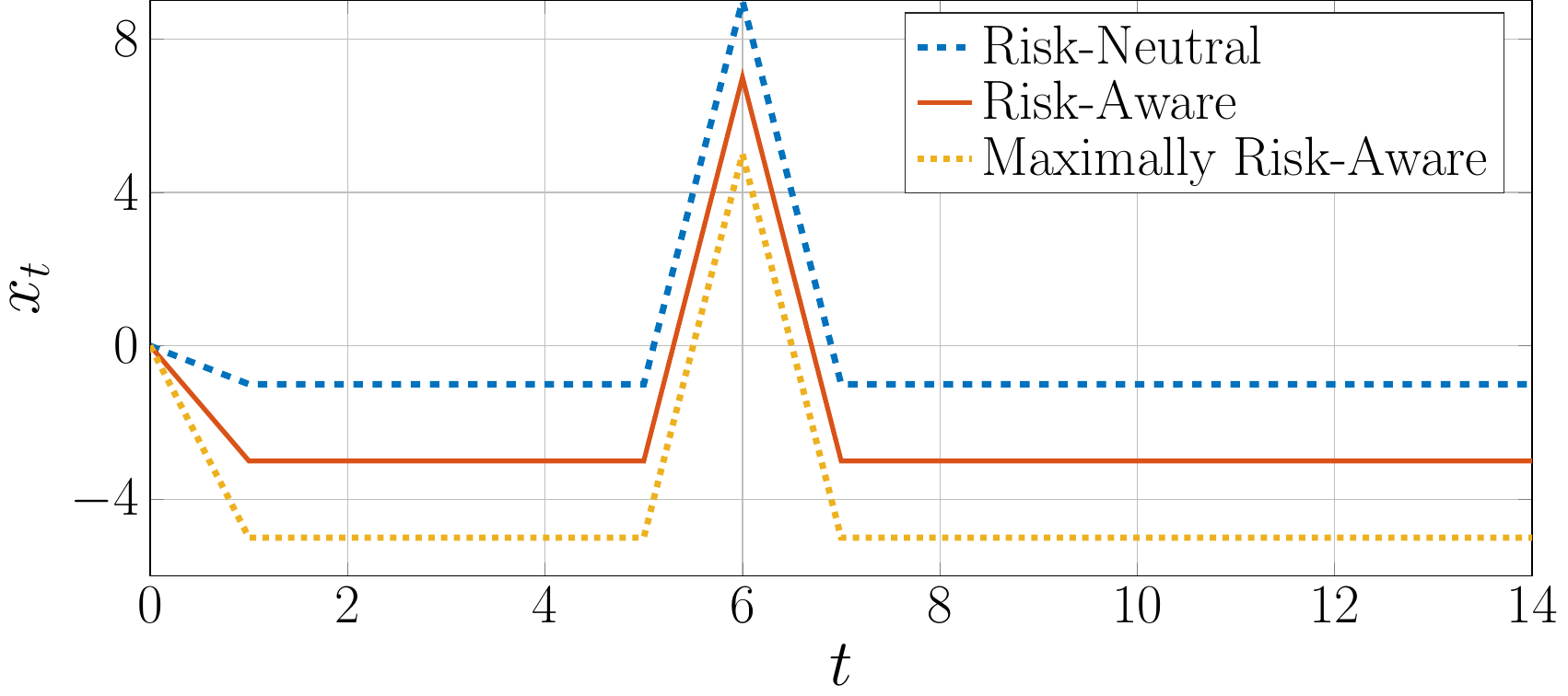}
	\vspace{-4bp}
	\caption{Comparison between risk-neutral and risk-aware control performance, when the system experiences rare but large shocks---here the shock occurs at time $6$. By sacrificing average behavior, the risk-aware controllers push the state away from the direction of the shock. 
}
	\label{fig:toy_example}
 	\vspace{-8bp}
\end{figure}

To highlight the usefulness of a risk-aware control policy, let us consider the following simple, motivating example. Let $x_{k+1}=x_{k}+u_k+w_{k+1}$ model an aerial robot, moving along a line. Assume that the process noise $w_k$ is \textit{i.i.d.} Bernoulli, taking the values $\beta>2$ with probability $1/\beta$ and $0$ with probability $1-1/\beta$. This noise represents shocks, e.g., wind gusts, that can occur with some small probability. We would then like to minimize the LQR cost $\E \sum_{t=0}^{N} \{x^2_t\}$, i.e., the total displacement of the robot over a horizon of $N$ time steps. In this case, the LQR optimal controller is $u^{\mathrm{LQR}}_k=-x_k-1$, where $-1\equiv-\mathbb{E}w_k$ cancels the mean of the process noise. We see that the LQR solution is \textit{risk-neutral}, as it does not account for the fact that the shock $\beta$ could be arbitrarily large.  On the other hand, the risk-aware LQR formulation proposed in this work results in a \textit{family} of optimal controllers of the form
\[
u^{*}_{t}(\mu)=-x_t-1-\frac{\mu}{1+2\mu}(\beta-2),\quad\mu\ge0,
\]
where $\mu$ controls the trade-off between average performance and risk. As $\mu$ increases, we move from the risk-neutral to the \textit{maximally risk-aware} controller $u^{*}_{t}(\infty)=-x_t-\beta/2$, which treats the noise as adversarial---see Fig.~\ref{fig:toy_example}.

In both classical and recent literature in linear-quadratic problems, risk awareness in estimation and control is typically achieved by replacing the respective random cost with its exponentiation~\cite{jacobson1973optimal,Whittle1981,bacsar2000risk,pham2012linear,roulet2019convergence,Speyer1992,Dey1999,Moore1997,Dey1997,Bauerle2014more}. Yet, the resulting stochastic control problem might not be well-defined for general classes of noise distributions, as it requires the moment generating function of the cost to be finite. Thus, heavy-tailed or skewed distributions, which are precisely those exhibiting high risk, are naturally excluded. Also, even if the expectation of the exponential cost is finite, it does not lead to a general, closed-form and interpretable solution. 
A notable exception is that of
Gaussian noise, also known as the Linear Exponential Quadratic Gaussian (LEQG) problem, which does enjoy a simple closed-form solution~\cite{Whittle1981,Speyer2008STOCHASTIC}. 
Apparently though, the Gaussian assumption is unable to capture distributions with asymmetric (skewed) structure, as in the above example.

Our contributions are as follows:

\noindent\textbf{--New Risk-Constrained Formulation.} We introduce a new risk-constrained formulation for the problem of LQ control in the case of partially-observed systems. The standard LQ objective is minimized subject to a total expected predictive variance risk constraint with respect to the state penalties. 
By tuning the risk constraint, we can trade between average performance and statistical variability of the state penalties.

\noindent\textbf{--General Noise Models.}
Contrary to the LEQG approach, our risk-constrained formulation is well-defined for general noise distributions, provided the associated fourth-order moments of the process noise are finite; thus, heavy-tailed or skewed noises are supported within our framework. For fully-observed systems, the optimal control law can be explicitly characterized under the same condition of finite fourth-order moments. In the case of general partially-observed systems, in order to characterize the optimal controller, we require the additional sufficient condition that all higher-order moments of the process noise exist. In any case, we do not require the existence of a moment generating function.

\noindent\textbf{--Characterization of Optimal Risk-Aware Controls.}
Under the aforementioned regularity conditions on the process noise, the constrained LQ problem admits a closed-form solution with a natural interpretation. The optimal risk-aware feedback controller is affine with respect to the optimal observer. 
The affine component
pushes the state away from directions where the state prediction error exhibits (skewed) heavy tails. Meanwhile, the state feedback gain satisfies a new risk-aware filtered Riccati recursion, in which the state penalty is inflated in riskier directions, where both the (conditional) covariance of the state prediction error and the state penalty are simultaneously larger. Interestingly, the separation principle holds, in the sense that the optimal observer is the minimum mean-square error estimator, which is designed independently of the control objective. To explicitly compute the parameters of the affine optimal control law, it is required to track several conditional moments, which might be a hard problem in general.

\noindent\textbf{--Explicit Risk-Aware LQR and LQG controllers.}
 In the special case of fully-observed systems (Linear Quadratic Regulator (LQR)) we can explicitly compute the optimal control law. The same is true for the case of partially-observed systems with Gaussian noise (Linear Quadratic Gaussian (LQG) control). Further, we show that our optimal risk-aware controllers are always stable, under standard controllability/observability conditions.
Interestingly, \textit{by appropriate re-parameterization}, our risk-aware LQR problem is equivalent to a generalized risk-neutral LQR problem with a tracking objective. Essentially, this implies that risk-neutral LQR formulations can provide inherent risk-averse behavior, as long as the involved parameters are selected in a principled way, as presented herein. A similar property holds for the risk-aware LQG problem.

\subsection{Related Work}
\textbf{\textit{Risk-aware optimization. }}
Risk-aware optimization has been studied in a wide variety of decision making contexts \cite{A.2018,Cardoso2019,W.Huang2017,Jiang2017,Kalogerias2018b,Tamar2017,Vitt2018,Zhou2018,Ruszczynski2010,SOPASAKIS2019281,chapman2019cvar,kalogerias2020noisy}. The basic idea is to replace expectations by more general functionals, called risk measures\cite{ShapiroLectures_2ND}, purposed to effectively quantify the statistical volatility of the involved random cost function, in addition to mean performance. Typical examples are mean-variance functionals \cite{Markowitz1952,ShapiroLectures_2ND},
mean-semideviations \cite{Kalogerias2018b}, and Conditional Value-at-Risk
(CVaR) \cite{Rockafellar1997}.

 \textbf{\textit{CVaR-optimal control.}}
In the case of control systems, CVaR optimization techniques have also been considered for risk-aware constraint satisfaction~\cite{chapman2019cvar}. Although CVaR captures variability and tail events well, CVaR optimization problems rarely enjoy closed-form expressions. Approximations are usually required to make computations tractable, e.g., process noise and controls are assumed to be finite-valued~\cite{chapman2019cvar}. Recently, in~\cite{chapman2022cvarLQ} a tractable upper bound was derived for the CVaR-LQR problem based on the assumption that the noise is finite-valued.

\textbf{\textit{Robust control.} }
Another related concept is that of robust control, where the system model or the noise profile is unknown~\cite{zhou1996robust,tzortzis2016robust,dean2020sample}. The objective is to optimally control the true system under worst case model uncertainty.
On the contrary, in risk-aware control, extreme noise events are part of the system model; they are not the outcome of model mismatch. Even if the system is exactly modeled, we would still need to consider risk-aware control if the process noise is heavy-tailed or highly variable. From this point of view, robustness and risk are complementary concepts.

\textbf{\textit{Mixed $\mathcal{H}_{2}$/$\mathcal{H}_{\infty}$ control. Regret-optimal control}}
Interestingly, there is a connection between mixed  $\mathcal{H}_{2}$/$\mathcal{H}_{\infty}$ control and risk-aware LEQG control~\cite{glover1988state,zhang2021derivative}. By increasing the exponential parameter in the LEQG control law, we trade average performance ($\mathcal{H}_{2}$) for closed-loop responses with smaller $\mathcal{H}_{\infty}$ norm. Another way to trade between robustness and performance was introduced in~\cite{goel2021regret}, where the worst-case regret with respect to non-causal $\mathcal{H}_2$ policies is minimized. 

 \textbf{\textit{Predictive variance.}}
Recently, in our previous work~\cite{tsiamis2020risk}, we introduced predictive variance as a new risk measure for LQR control and used it in a risk-constrained optimal control formulation. The results of~\cite{tsiamis2020risk} were extended to the infinite horizon case in~\cite{zhao2021infinite}. The performance of the policy gradient algorithm in the case of risk-constrained Linear Quadratic Regulators was also studied in~\cite{zhao2021global}.
Predictive variance constraints have also been used as a measure of risk in portfolio optimization~\cite{abeille2016lqg}; different from our paper, the noise is limited to Gaussian distributions and the variance is with respect to linear stage costs. Note that our previous work~\cite{tsiamis2020risk} contains only preliminary results for fully-observed systems. Here, we study the more general and challenging problem of LQ control in the case of partially observed systems. In fact, the optimal feedback law in the case of partially-observed systems can be quite different from the feedback law in the fully-observed case, even in the case of Gaussian noise--see Section~\ref{Section_Gaussian} for more details.

\textit{\textbf{Notation and Structure:}}
The transpose operation is denoted by $(\cdot)'$. If $x_{k},\dots,x_{t}$ is a sequence of vectors, then $x_{k:t}$ denotes the batch vector of all $x_i$ for $k\le i\le t$. We use the notation $\snorm{\cdot}_2$ to denote both the square norm of vectors and the spectral norm of matrices. The $\sigma$-algebra generated by a random vector $x$ is denoted by $\sigma(x)$. By $\LL_p(\F)$, we denote the space of $\F$-measurable random variables (vectors) with finite $p$-order moments. 
The remaining paper is structured as follows. In Section~\ref{Section_Formulation}, we introduce our risk-aware formulation of LQ control. In Sections~\ref{sec:reformulation},~\ref{sec:duality} we show that our the risk-aware LQ problem can be reformulated as a Quadratically Constrained Quadratic Problem and solved by exploiting Lagrangian duality. In Sections~\ref{Section_Fully_Observed},~\ref{Section_Gaussian}, we provide explicit control laws for the problem of risk-aware LQR and risk-aware LQG control respectively. In Section~\ref{Section_Partially}, we characterize the optimal control laws in the case of general partially observed systems. We conclude with numerical simulations in Section~\ref{Section_Simulations} and with remarks in Section~\ref{sec:conclusion}.
\section{Risk-Constrained LQ Formulation} \label{Section_Formulation}

Consider system~\eqref{FOR_EQN_system},
where $x_t\in \R^n$ is the state, $u_t\in \R^p$ is the control signal, and $y_t\in\R^m$ is the measured output. Matrix $A\in \R^{n\times n}$ is the state transition matrix, $B\in \R^{n\times p}$ is the input matrix, and $C\in \R^{m\times n}$ is the output matrix. We assume that the initial value $x_0$ is deterministic and fixed. Signal $w_{t}\in \R^n$ is a random process noise, while $v_t\in \R^m$ is a random measurement noise. The process $(w_t,v_t)$ is assumed to be \textit{i.i.d} across time, but it can have any joint distribution (possibly non-Gaussian). 
For $t\ge0$, let $\F_t=\sigma\paren{y_{0:t},u_{0:t}}$ be the $\sigma$-algebra generated by all observables up to time $t$, and let $\F_{-1}$ be the trivial $\sigma$-algebra. Based on this notation, $u_t$ is $\F_t$-measurable, while $(w_{t+1},v_{t+1})$ is independent of $\F_t$.
We also make an additional assumption on the process noise.
\begin{assumption}[\textbf{Noise Regularity}]\label{FOR_ASS_noise}
The process $w_t$ has finite fourth-order moment, i.e., for every $t\in \mathbb{N}$, $\E\norm{w_{t}}^4_2<\infty$.
\end{assumption}
\noindent The above mild regularity condition is required for our risk measure to be well-defined. It is satisfied by general noise distributions, including many heavy-tailed ones.
Denote the mean of the noise by $\bar{w}\triangleq\E w_k$ and its variance by $W\triangleq \E(w_k-\bar{w})(w_k-\bar{w})'$.

As discussed in Section~\ref{Section_Introduction}, the classical LQ problem is risk-neutral, since it optimizes performance only on average~\cite{bertsekas2017dynamic}. 
Still, even if average performance is good, the state can grow arbitrarily large under less probable, yet extreme events.
In other words, the state can exhibit large variability.
To deal with this issue, we propose a risk-constrained formulation of the LQ control problem, posed as
\begin{equation}\label{eq:LQ_constrained}
\begin{aligned}
\min_{u} &\quad \E\set{ x'_NQ x_N+\sum_{t=0}^{N-1} x'_tQx_t+u'_tRu_t}\\
\mathrm{s.t.} &\quad \E \set{\sum_{t=1}^{N} \left[ x'_tQ x_t-\E\paren{x'_tQx_t|\F_{t-1}} \right]^2}\le \epsilon\\
&\quad \text{Dynamics~\eqref{FOR_EQN_system}}\\ 
&\quad u_t\in \LL_{4}(\F_t),\,t=0,\dots, N-1
\end{aligned}\,,
\end{equation}
where $u=u_{0:N-1}$ are the inputs from time $0$ up to time $N-1$, for some horizon $N\in \mathbb{N}$. For each $t$, the \textit{causality constraint} on $u_t$ restricts the inputs to the space of $\F_t$-measurable random vectors of appropriate dimension with bounded fourth-order moments,
denoted as $\LL_4(\F_t)$. 
Here, the risk measure adopted is the \emph{(cumulative expected) predictive variance} of the state cost. The predictive variance incorporates information about the tail \textit{and} skewness of the penalty $x_t'Qx_t$. This forces the controller to take higher-order noise statistics into account, mitigating the effect of rare though large noise values.
Hence, our risk-aware LQ formulation not only forces the state $x_t$ to be close to zero, but also explicitly restricts its variability. The initial state is fixed (for simplicity), so there is no associated risk term for $t=0$. The fourth-order integrability constraint on the inputs along with Assumption~\ref{FOR_ASS_noise} are sufficient to guarantee that the cumulative expected predictive variance is well-defined.

\begin{remark}[\textbf{Input integrability}]
The fourth-order integrability condition $u_t\in\LL_4(\F_t)$ on the inputs is stricter compared with the risk-neutral formulation, where only square-integrability is needed. In the general case of partially-observed systems, this condition is needed to guarantee that the constraint in~\eqref{eq:LQ_constrained} is well-defined. However, in many cases of interest, this condition is not essential. For example, in the fully observed case (Section~\ref{Section_Fully_Observed}), we can pose problem~\eqref{eq:LQ_constrained} with the constraint $u_t\in \LL_{2}(\F_t)$ and the optimal control is still guaranteed to be in $\LL_{4}(\F_t)$; this is a byproduct of the noise regularity Assumption~\ref{FOR_ASS_noise}. The same holds for the case of partially-observed systems with Gaussian noise (Section~\ref{Section_Gaussian}). \hfill $\diamond$
\end{remark}

Problem \eqref{eq:LQ_constrained} offers a simple and interpretable way to control the trade-off between average performance and risk. By simply decreasing $\epsilon$, we increase risk-awareness. Inspired by standard risk-aware formulations, in the above optimization problem our risk definition is tied to the specific state penalty $x'_tQx_t$. However, all of our results are still valid if we employ the predictive variance of a different quadratic form, e.g., the norm of the state, $\norm{x_t}_2^2$, in the constraint. 
 In the following sections, we characterize the optimal controllers in the case of general partially-observed systems. We also provide explicit, finite-dimensional control laws for the case of i) fully-observed systems with general noise, which we term risk-aware LQR controllers; and ii) partially-observed systems with Gaussian noise, which we term risk-aware LQG controllers.

\section{Quadratic Reformulation of Risk-Constrained LQ Control}\label{sec:reformulation}
The solution procedure of the risk-aware dynamic program~\eqref{eq:LQ_constrained} consists of the following steps. First, we ensure the well-definiteness of~\eqref{eq:LQ_constrained}, also showing that~\eqref{eq:LQ_constrained} can be equivalently expressed as a \textit{sequential variational Quadratically Contrained Quadratic Program (QCQP)}, or, more precisely, as a \textit{Quadratically Constrained LQ (QC-LQ)} problem (Proposition~\ref{prop:partially_existence}). Then, we exploit Lagrangian duality (Theorem~\ref{thm:KKT}) to solve~\eqref{eq:LQ_constrained} \textit{exactly and in closed form}. More specifically, we first derive an explicit expression for the optimal risk-aware controller~(Theorems~\ref{thm:optimal_LQR},~\ref{thm:partially_observed},~\ref{thm:gaussian}), given an arbitrary but fixed Lagrange multiplier. Then, we show how an optimal Lagrange multiplier may be efficiently discovered via trivial bisection~(Theorem~\ref{thm:OPTIMAL}).

Since we are dealing with partially observed systems, we can only approximately estimate the current state $x_t$ based on the information $\F_t$ collected so far. Define the state estimate and the state prediction at time $t$ respectively as
\begin{align*}
   \hat{x}_{t|t}&=\E(x_t|\F_t) \\
   \hat{x}_{t}&=\E(x_t|\F_{t-1}).
\end{align*} 
Note that both values are mean-square optimal, i.e. they minimize the mean square estimation error (prediction error respectively)~\cite{anderson2005optimal}.  Under Assumption~\ref{FOR_ASS_noise} on $w_t$, and since $u_t\in\LL_4(\F_{t})$ both expectations are well-defined. The state prediction and the state estimate are related via the expression
\[
\hat{x}_{t}=A\hat{x}_{t-1|t-1}+Bu_{t-1}+\bar{w}.
\]
The innovation (or prediction) error is defined as
\begin{equation}\label{eq:prediction_error}
\delta_t\triangleq x_t-\hat{x}_{t}.
\end{equation}
Define also the refinement error between the prediction and the estimate:
\begin{equation}\label{eq:estimation_error}
	e_t\triangleq \hat{x}_{t|t}-\hat{x}_{t}.
\end{equation}
Both errors are martingale differences, satisfying the mean conditions $\E (e_t|\F_{t-1})=0,\E (\delta_t|\F_{t-1})=0$.
Note that in the general case of non-Gaussian noise, the innovation error $\delta_t$ is not i.i.d. and not independent of the past in general. 

In the following result, we show that the predictive variance constraint has an underlying quadratic structure.
\begin{proposition}[\textbf{Quadratic Reformulation}]\label{prop:partially_existence}
Let Assumption~\ref{FOR_ASS_noise} be in effect and define the (random) conditional moments
\begin{align*}
     W_{t-1}&\triangleq \E(\delta_{t}\delta'_{t}|\F_{t-1})\\
     m_{3,t-1}&\triangleq 2Q\E\set{\delta_{t}\delta_t'Q\delta_t|\F_{t-1}}\\
     m_{4,t-1}&\triangleq \E\set{	(\delta_t'Q\delta_t -\Tr(QW_{t-1}))^2|\F_{t-1}}.
\end{align*}
Then, the risk-constrained LQ problem~\eqref{eq:LQ_constrained} is well-defined and equivalent to the sequential variational QCQP
\begin{align}
\min_{u}&\hspace{-6pt}& J(u) \triangleq\,\,& \E\set{ x'_NQ x_N+\sum_{t=0}^{N-1} x'_tQx_t+u'_tRu_t} \label{eq:LQ_constrained_reformulated}\\
\mathrm{s.t.} &\hspace{-6pt}& J_R(u) \triangleq\,\,& \E\set{\sum_{t=1}^{N} 4\hat{x}'_{t}QW_{t-1}Q\hat{x}_{t}+2\hat{x}'_{t}m_{3,t-1}}\le \bar{\epsilon}\nonumber\\
&\hspace{-6pt}& & \text{\textnormal{Dynamics}~\eqref{FOR_EQN_system}}\nonumber\\ 
&\hspace{-6pt}& & u_t\in \LL_{4}(\F_t),\,t=0,\dots, N-1,\nonumber
\end{align}
where
$
 \bar{\epsilon}\triangleq \epsilon-\sum_{t=1}^{N} \E m_{4,t-1}.
$
\end{proposition}
\begin{proof}
Define the state penalty difference:
\begin{equation}\label{eq:Delta}
	\Delta_t\triangleq x'_tQ x_t-\E\paren{x'_tQx_t|\F_{t-1}}.
\end{equation}
Since the inputs $u_t\in\LL_4$ and the disturbances $w_t\in\LL_4$ have finite fourth moments, it follows that $x_t\in\LL_4$ since it is a linear combination of inputs and disturbances. 
As a result, $x'_tQ x_t$ is integrable, and $\Delta_t$ is well defined.
Next, we find an expression for $\Delta^2_t$.
By the definition of $\hat{x}_t,\delta_t$
\[
x_t=\hat{x}_{t}+\delta_t,
\]
where the prediction $\hat{x}_{t}\in\LL_4$ is well-defined since $x_t\in\LL_4$. Similarly $\delta_t\in\LL_4$.
Based on the above decomposition, the quadratic form becomes
\[
x'_tQx_t=\hat{x}_{t}'Q\hat{x}_{t}+2\hat{x}_{t}'Q\delta_t+\delta_t'Q\delta_t.
\]
All three terms are integrable since $x_t\in\LL_4$ and $\delta_t\in\LL_4$.
By orthogonality, the cross terms have zero expected value:
\[
\E (\hat{x}_{t}'Q\delta_t|\F_{t-1})= \hat{x}_{t}'Q(\E(\delta_t|\F_{t-1}))=0.
\]
This implies that
\[
\E (x'_tQx_t|\F_{t-1})=\hat{x}_{t}'Q\hat{x}_{t}+\E(\delta_t'Q\delta_t|\F_{t-1})
\]
As a result, we obtain the expression
\[
	\Delta_t=\delta_t'Q\delta_t -\Tr(QW_{t-1})+2\hat{x}'_{t}Q\delta_{t},
\]
which leads to
\begin{equation*}
	\begin{aligned}
		\Delta^2_t&=(\delta_t'Q\delta_t -\Tr(QW_{t-1}))^2+4\hat{x}'_{t}Q\delta_{t}\delta'_{t}Q\hat{x}_{t}\\&\quad\,+4\hat{x}'_{t}Q\delta_{t}(\delta_t'Q\delta_t -\Tr(QW_{t-1})).
	\end{aligned}
\end{equation*}
Finally, we show that $\Delta^2_t$ are integrable.
Integrability of all terms follows from the existence of the fourth moments of $\delta_t,\,u_t,\,\hat{x}_t$ and Hölder's inequality
\[
\E \snorm{\alpha}^p_2 \snorm{\beta}^q_2 \le (\E \snorm{\alpha}^4_2)^{p/4}(\E \snorm{\beta}^4_2)^{q/4},
\]
for $p+q=4$, $p,q\ge 0$.
Hence, the total expected predictive-variance $\E\sum_{t=0}^{N-1}\Delta^2_t$ is well-defined. Moreover, we have
\begin{align*}
    	\E\set{\Delta_t^2|\F_{t-1}}&=4\hat{x}'_{t}Q W_{t-1}Q\hat{x}_{t}+ 2\hat{x}'_{t}m_{3,t-1}+m_{4,t-1}.
\end{align*}
To complete the proof, we take the expectation and move the $m_{4,t-1}$ terms to the right-hand side of the constraint.
\end{proof}
The above reformulation enables us to apply duality theory, as discussed next. Note that the equivalent constraint is quadratic. The quadratic and linear penalties $W_{t-1},m_{3,t-1}$ are random variables and depend on the observations up to time $t-1$.  If the prediction error $\delta_t$ is independent of the past $\F_{t-1}$, e.g. in the special case of fully observed systems or Gaussian noise, then the penalties are deterministic and the above expressions can be simplified. 
\section{Lagrangian Duality}\label{sec:duality}
To tackle problem~\eqref{eq:LQ_constrained}, we now consider the \textit{variational Lagrangian} $\LAG : \LL_{2}(\F_0) \times \dots \times \LL_{2}(\F_{N-1}) \times \mathbb{R}_+ \rightarrow \mathbb{R}$ of the sequential QCQP~\eqref{eq:LQ_constrained_reformulated}, defined as
\begin{align}\label{eq:Lagrangian}
&\LAG(u,\mu)\triangleq J(u)+\mu J_R(u)-\mu \bar{\epsilon},
\end{align}
where $\mu\in\mathbb{R}_+$ is a multiplier associated with the variational risk constraint of \eqref{eq:LQ_constrained_reformulated}. Hereafter, problem \eqref{eq:LQ_constrained_reformulated} will be called the \textit{primal problem}.
Accordingly, the \textit{dual function} $D\hspace{-0.5pt}\hspace{-0.5pt}:\hspace{-0.5pt}\hspace{-0.5pt}\mathbb{R}_{+}\hspace{-0.5pt}\hspace{-0.5pt}\rightarrow\hspace{-0.5pt}\hspace{-0.5pt}\hspace{-0.5pt}[-\infty,\infty)$
is additionally defined as
\begin{equation}\label{eq:FDUAL}
D(\mu)\triangleq\inf_{u\in{\cal U}_0}\mathpzc{L}(u,\mu),
\end{equation}
where the \textit{implicit feasible set} $\mathcal{U}_0$ obeys ($k\le N-1$)
\begin{equation}
\mathcal{U}_k \Neg{1} \triangleq \Neg{1.5} \set{\Neg{.5} u_{k:N-1} \Neg{1}\in\Neg{1} \prod_{t=k}^{N-1} \LL_{4}(\F_t) \Neg{-1}\Bigg|\Neg{-1} \begin{aligned}x_{t+1}\Neg{1}&=\Neg{1}Ax_t\Neg{1}+\Neg{1}Bu_t\Neg{1}+\Neg{1}w_{t+1}\\
y_t\Neg{1}&=\Neg{1}Cx_t\Neg{1}+\Neg{1}v_t\Neg{1}\end{aligned}\Neg{.5}}\Neg{1}\Neg{.5},\Neg{1}\nonumber
\end{equation}
and contains the constraints of \eqref{eq:LQ_constrained_reformulated} that have not been dualized in the construction of the Lagrangian in \eqref{eq:Lagrangian}.
Note that it is always the case that $D\le J^{*}$ on $\mathbb{R}_{+}$, where  $J^{*}\hspace{-1pt}\in\hspace{-1pt}[0,\infty]$
denotes the optimal value of the primal problem \eqref{eq:LQ_constrained_reformulated}.
Then, the optimal value of the always concave \textit{dual problem}
\begin{equation}\label{eq:dual}
\sup_{\mu\ge0} D(\mu) \equiv \sup_{\mu\ge0}\inf_{u\in{\cal U}_0}\mathpzc{L}(u,\mu),
\end{equation}
${D}^{*}\hspace{-2pt}\triangleq\hspace{-.5pt}\sup_{\mu\ge0}{D}(\mu)\hspace{-1pt}\in\hspace{-1pt}[-\infty,\hspace{-0.5pt}\infty]$,
is the tightest under-estimate of ${J}^{*}$, when knowing
only ${D}$.

Leveraging Lagrangian duality, we may now state the following result, which provides sufficient optimality conditions for the QCQP \eqref{eq:LQ_constrained_reformulated}. The proof is omitted, as it follows as direct application of \cite[Theorem 4.10]{ruszczynski2011nonlinear}.

\begin{theorem}[\textbf{Optimality Conditions}]\label{thm:KKT}
Let Assumption \ref{FOR_ASS_noise} be in effect. Suppose that there exists a feasible policy-multiplier pair $(u^*,\mu^*) \in \mathcal{U}_0 \times \mathbb{R}_+$ such that
\begin{enumerate}
    \item $\LAG(u^*(\mu^*),\mu^*)=\min_{u\in\mathcal{U}_0}\LAG(u,\mu^*)=D(\mu^*)$;
    
    \item $J_R(u^*)\le\bar{\epsilon}$, i.e., the dualized variational risk constraint of \eqref{eq:LQ_constrained_reformulated} is satisfied by control policy $u^*$;
    
    \item $\mu^*\Neg{.5}(J_R(u^*)-\bar{\epsilon})\Neg{2.1}=\Neg{1.8}0$, i.e., complementary slackness holds.
\end{enumerate}
Then, $u^*$ is optimal for both the primal problem \eqref{eq:LQ_constrained_reformulated} and the initial problem \eqref{eq:LQ_constrained}, $\mu^*$ is optimal for the dual problem \eqref{eq:dual}, and \eqref{eq:LQ_constrained_reformulated} exhibits zero duality gap, that is, $D^*\equiv P^*<\infty$.
\end{theorem}

Theorem \ref{thm:KKT} will be serving as the backbone of our analysis towards the solution to problem \eqref{eq:LQ_constrained_reformulated}. It is sufficient to compute the relaxed optimal input $u^*(\mu)$ of the Lagrangian in~\eqref{eq:FDUAL}, for any given multiplier $\mu\ge0$. Then, we can also compute an optimal multiplier $\mu^*$ via bisection, thus providing a complete solution to the primal problem.
The use of bisection is based on the following theorem (the proof can be found in the Appendix).
\begin{theorem}[\textbf{Optimal Multiplier}]\label{thm:OPTIMAL} Let Assumption \ref{FOR_ASS_noise} be in effect. Assume that for any $\mu\ge 0$ the minimum in~\eqref{eq:FDUAL} is attained by
a control policy $u^*(\mu)\in \mathcal{U}_0$. Assume that the risk constraint functional $J_R(u^*(\cdot))$ is a continuous function of $\mu$. Then, the following statements are true:
\begin{enumerate}
    \item The LQ cost $J(u^*(\mu))$ is increasing with $\mu\ge0$, while the risk constraint functional $J_R(u^*(\mu))$ is decreasing.
    \item
    Define the multiplier  
    \begin{align}\label{eq:Optimal_Multiplier}
    \mu^* \triangleq \inf\set{\mu \ge 0:\:J_R(u^*(\mu))\le \bar{\epsilon}}.
\end{align}
If $\mu^*$ is finite, then the policy $u^*(\mu^*)$ is optimal for the primal problem~\eqref{eq:LQ_constrained_reformulated}, and this is the case as long as 
\eqref{eq:LQ_constrained_reformulated} 
 satisfies
Slater's condition:
\[
J_R(u^{\dagger})<\bar{\epsilon},\,\text{ for some }u^{\dagger}\in \mathcal{U}_0.
\]
\end{enumerate} 
\end{theorem}
The above result exploits the fact that, under the relaxed optimal policy $u^*(\cdot)$, both the LQ cost $J(u^*(\cdot))$ and the risk  functional $J_R(u^*(\cdot))$ are monotone functions.
Note that in order to apply Theorem~\ref{thm:OPTIMAL}, we need to verify three conditions i) existence of an optimal solution~$u^*(\mu)$, ii) continuity of $J_R(u^*(\mu))$, and iii) satisfaction of Slater's condition. This is the subject of the following sections.
\section{Optimal Risk-Aware LQR Control}\label{Section_Fully_Observed}
Let us study first the simpler case of fully-observed systems, where $y_k=x_k$, i.e. there is no measurement noise $v_k=0$ and the output matrix is the identity $C=I$. This problem is the risk-constrained version of the classical Linear Quadratic Regulator (LQR) problem. In this case, the conditional moments in Proposition~\ref{prop:partially_existence} can be simplified significantly leading to an optimal control law which is easy to interpret, providing intuition for the solution of the general risk-aware LQ problem.

Let $\mu\ge0$ be arbitrary but fixed. First, we may simplify the form of the Lagrangian $\LAG$ and express it within a canonical dynamic programming framework. In this respect, we have the following straightforward, but key result.
\begin{lemma}[\textbf{Lagrangian Reformulation}]\label{lem:Lagrangian_Dynamic_Programming}
Assume that system~\eqref{FOR_EQN_system} is fully-observed: $y_k=x_k$ for all $k\ge 0$. Let Assumption~\ref{FOR_ASS_noise} be in effect. Consider the sequential variational QCQP problem~\eqref{eq:LQ_constrained_reformulated}.  Consider the notation of Proposition~\ref{prop:partially_existence}. 
Define the inflated state penalty matrix
\begin{equation}
    Q_{\mu} \triangleq Q+4\mu QWQ.\nonumber
\end{equation}
Then the innovation process $\delta_k=w_k-\bar{w}$ is i.i.d. and independent of $\F_{k-1}$ with
\begin{align*}
    W_{t-1}&=W,\quad
    m_{3,t-1}=m_3\triangleq 2Q\E\set{\delta_{t}\delta_t'Q\delta_t}\\
    m_{4,t-1}&=m_4\triangleq \E\set{	(\delta_t'Q\delta_t -\Tr(QW))^2}.
\end{align*}
Moreover, for every $u_t\in\LL_4(\F_t)$, $t\le N-1$, the Lagrangian function $\LAG$ can be expressed as
	\begin{equation}\label{ANA_EQN_Lagrangian}
	\LL(u,\mu)\hspace{-1pt}=\hspace{-1pt}  \E\hspace{-1pt}\set{\sum_{t=1}^{N}g_{t}(x_t,u_{t-1},\mu)\hspace{-1pt}}\hspace{-.5pt}+g(\mu),\hspace{-1pt}
	\end{equation}
	where
	\begin{align*}
		g_t(x_t,u_{t-1},\mu)&\triangleq x_t'Q_{\mu}x_t+2\mu m_3'x_t+u_{t-1}'Ru_{t-1},\, t\le N\\
g(\mu)&\triangleq \mu\paren{-\bar{\epsilon}-4N\Tr{(WQ)^2}}+x_0'Qx_0.
	\end{align*}
\end{lemma}
\begin{proof}
The  properties of $\delta_t$ follow immediately from~\eqref{FOR_EQN_system}, full observability, and the fact that $w_k$ is i.i.d. As a result, all moments $W_t,m_{3,t},m_{4,t}$ are deterministic and constant over time.
For the Lagrangian reformulation, we used Proposition~\ref{prop:partially_existence}, the form of $\LAG$, and the identities
\begin{align*}
\E (\hat{x}'_k QWQ\hat{x}_k)&=\E (x_k QWQx_k)-\E (\delta'_k QWQ\delta_k)\\
\E(\delta'_k m_3)&=0,\, \E(\hat{x}'_k m_3)=\E(x'_k m_3).\qedhere
\end{align*}
\end{proof}
\begin{remark}[\textbf{Relation to LQR with tracking}]\label{ANA_REM_LQR}
The Lagrangian~\eqref{ANA_EQN_Lagrangian} has the structure of a generalized LQR problem with a tracking objective. Substituting for $m_3=QM_3$, where
\[
M_3\triangleq 2\E\set{\delta_t\delta'_tQ \delta_t},
\]
 we can rewrite the stage cost as
\begin{align*}
    g_{t}(x_t,u_t,\mu)=(x_t+\mu M_3)'Q(x_t+\mu M_3)\Neg{-12}\\
+x'_t(4\mu QWQ)x_t+u_t'Ru_t-\mu^2 M_3'QM_3,
\end{align*}
i.e., the state penalty is quadratic and consists of two distinct terms. The first one, i.e., $(x_t+\mu M_3)'Q(x_t+\mu M_3)$ is a tracking error term that forces the state to be close to the static target $-\mu M_3$. Informally, in the case of skewed noise, by tracking $-\mu M_3$ we pre-compensate for directions in which the distribution of the noise has heavy tails. This decreases the statistical variability of the predicted stage cost.  The second term, $x_t'(4\mu QWQ)x_t$, is a standard quadratic penalty term;  notice that, contrary to the risk-neutral case, the covariance of the noise $W$ now affects the penalty term. Informally, this term penalizes state directions which not only lead to high cost but are also more sensitive to noise, as captured by the product $QWQ$. 
Hence, the risk-neutral LQR framework can exhibit inherent risk-averse properties, provided that its parameters are selected in a principled way. Of course, selecting those parameters \emph{a priori} is not trivial.\hfill $\diamond$
\end{remark}

The structure of the Lagrangian as suggested by Lemma~\ref{lem:Lagrangian_Dynamic_Programming} enables us to derive both a closed-form expression for its minimum 
and an explicit optimal control policy. To this end, define the \textit{optimal cost-to-go} at stage $k\le N-1$ as
\begin{equation}
\begin{aligned}
& \LAG^*_k(x_k,\mu) \triangleq\inf_{u_{k:N-1} \in  \mathcal{U}_k} \E \set{\sum_{t=k}^{N-1}g_{t+1}(x_{t+1},u_{t},\mu)\Bigg|\F_k},\nonumber
\end{aligned}
\end{equation}
where we omit the constant components of the Lagrangian. Under this definition, it is true that
\[
D(\mu)\equiv\inf_{u\in\mathcal{U}_0} \LAG(u,\mu)=\LAG^*_0(x_0,\mu)+g(\mu).
\]
We may now derive the complete solution to \eqref{eq:FDUAL}, which provides optimal risk-aware control policies for every multiplier $\mu\ge0$.

\begin{theorem}[\textbf{LQR Risk-Aware Controllers}]\label{ANA_THM_Optimal_Input}
    Assume that system~\eqref{FOR_EQN_system} is fully-observed: $y_k=x_k$ for all $k\ge 0$. Let Assumption~\ref{FOR_ASS_noise} be in effect, choose $\mu\ge0$, and adopt the notation of Lemma \ref{lem:Lagrangian_Dynamic_Programming}. For $t\le N-1$, the optimal cost-to-go $\LAG^*_t(x_t,\mu)$ may be expressed as
\begin{align*}
\LAG^*_t(x_t,\mu)= x'_{t}(V_{t}-Q_{\mu})x_{t} +2(\xi_t-\mu m_{3})' x_{t}+c_{t},
\end{align*}
where the quantities $V_t$, $\xi_t$, and $c_t$ are evaluated through the backward recursions
	\begin{align}
\label{eq:Riccati_Recursions_Fully}
V_{t-1}&\Neg{1}=\Neg{1}A'V_{t}A\hspace{-.5bp}+\hspace{-.5bp}Q_{\mu}\hspace{-1.5bp}-\hspace{-1.5bp}A'V_{t}B(B'V_{t}B\hspace{-.5bp}+\hspace{-.5bp}R)^{-1}B'V_{t}A,\\
K_{t-1}&\Neg{1}=\Neg{1}-(B'V_{t}B+R)^{-1}B'V_{t}A,\\
\xi_{t-1}	&\Neg{1}=\Neg{1}(A+BK_{t-1})'(\xi_{t}+V_{t}\bar{w})+\mu m_{3},\\
l_{t-1}&\Neg{1}=-(B'V_{t}B+R)^{-1}B'(\xi_t+V_t \bar{w}),\\
c_{t-1}&\Neg{1}=\Neg{0.5}c_t+\Tr (WV_{t})+2\xi'_t\bar{w}+\bar{w}'V_t\bar{w}\nonumber \\
&\Neg{-6}-l_{t-1}'(B'V_tB+R)l_{t-1},
\end{align}
with terminal values $V_N=Q_{\mu}$, $\xi_N=\mu m_{3}$, and $c_N=0$.
Additionally, an optimal control policy 
that achieves the dual value in \eqref{eq:FDUAL} may be expressed as
		\begin{align}\label{eq:LQR_Optimal_Input}
u^*_{t}(\mu)=K_{t}x_{t}+l_{t} \in  \LL_4(\F_t),\,\,\,\forall t\le N-1,
	\end{align}
and is unique up to sets of probability measure zero. 
\end{theorem}
\begin{proof}
The proof is similar to that of Theorem~\ref{thm:partially_observed} in Section~\ref{Section_Partially} and is, thus, omitted.  The only difference is that we need to verify that the input has bounded fourth moments $u^*_t(\mu)\in  \LL_4(\F_t)$ under Assumption~\ref{FOR_ASS_noise}. This can be inferred recursively by~\eqref{eq:LQR_Optimal_Input} and by the fact that $K_t,l_t$ are deterministic constants for all $t\ge 0$ (at all time steps the input is a linear combination of random variables with bounded fourth moments). 
\end{proof}

As suggested by Remark~\ref{ANA_REM_LQR}, it turns out that the optimal controller~\eqref{eq:LQR_Optimal_Input} is affine with respect to the state. If we expand $\xi_t$, we can see that the affine term $\ell_t$ consists of two components:
\[
l_t= -(B'V_{t}B+R)^{-1}B'(S_t \mu m_3+T_t \bar{w}),
\]
for some appropriate matrices $S_t,T_t$:
\begin{align*}
S_t&=(A+BK_t)'S_{t+1}+I,\quad S_N=I\\
T_t&=(A+BK_t)'(T_{t+1}+V_{t+1}),\quad T_N=0.
\end{align*}
One component forces the state to track the reference $-\mu m_3$, which points away from 
heavy-tailed regions of the noise distribution. The other component acts against the mean value of the noise--such a term also appears in risk-neutral LQR.
Meanwhile, the state-feedback term accounts for the internal dynamics. Similar to the risk-neutral case, the controller's behavior is governed by a Riccati difference equation~\eqref{eq:Riccati_Recursions_Fully}. However, we now have an inflated stage cost matrix $Q_{\mu}=Q+4\mu QWQ$, instead of the original.  As suggested by the product $QWQ$, the risk-aware control gain becomes more strict in directions that are simultaneously more costly and prone to noise, as captured by the covariance $W$. As a sanity check, we can verify that for $\mu=0$, we recover the risk-neutral LQR optimal controller, i.e. $Q_{0}=Q$ and $l_t$ depends only on the mean value of the noise $\bar{w}$

Since $V_t$ in~\eqref{eq:Riccati_Recursions_Fully} satisfies a standard Riccati difference equation with an inflated matrix $Q_{\mu}$, we immediately obtain from standard LQR theory that for any $\mu\ge 0$, the optimal controller~\eqref{eq:LQR_Optimal_Input} will be internally stable. Matrix $A+BK_t$ will converge and its spectral radius will eventually be bounded as $\rho(A+BK_t)<1$, as the horizon $N$ grows to infinity. The conditions for stability remain the same as in risk-neutral LQR.
\begin{assumption}[\textbf{Controllability}]\label{ass:controllability}
The pair $(A,B)$ is stabilizable, the pair $(A,Q^{1/2})$ is detectable, matrix $Q\succeq 0$ is positive semi-definite and matrix $R\succ 0$ is positive definite.
\end{assumption}
\begin{corollary}[\textbf{Internal Stability}]\label{cor:stability}
Let Assumptions \ref{FOR_ASS_noise} and \ref{ass:controllability} be in effect, and adopt the notation of Lemma \ref{lem:Lagrangian_Dynamic_Programming}.
For fixed $\mu\ge0$, consider the control policy $u^*(\mu)$, as defined in~\eqref{eq:LQR_Optimal_Input}. As $N\rightarrow \infty$, $V_t$ converges exponentially fast to the unique stabilizing solution\footnote{A stabilizing solution renders $A+BK$ stable.} of the algebraic Riccati equation
\[
V=A'VA+Q_{\mu}- A'VB(B'VB+R)^{-1}B'VA.
\] 
As a result, for every $t\ge 0$, it is true that, as $N\rightarrow \infty$, 
\begin{align*}
    K_t&\rightarrow K\triangleq -(B'VB+R)^{-1}B'VA,\\
    \xi_t&\rightarrow \xi \triangleq (I-(A+BK)')^{-1}\set{(A+BK)'V\bar{w}+\mu m_3},\\
    l_t&\rightarrow l \triangleq -(B'VB+R)^{-1}B'(\xi+V\bar{w}),
\end{align*}
exponentially fast, and the closed-loop matrix $A+BK$ is stable (spectral radius $\rho(A+BK)<1$).
\end{corollary}
\begin{proof}
Since $Q_{\mu}\succeq Q$ and $(A,Q^{1/2})$ is detectable, the pair $(A,Q_{\mu}^{1/2})$ is also detectable.
Since $(A,B)$ is stabilizable, $(A,Q_{\mu}^{1/2})$ is detectable, and $R\succ 0$, the exponential convergence of $V_t$ and $K_t$ to $V$ and $K$ respectively, and the stability of $A+BK$ follow from standard LQR theory~\cite[Chapter 4]{anderson2005optimal}.  The proof of the convergence of the remaining terms follows similar steps.
\end{proof}
\subsection{Recovery of Primal-Optimal Solutions}
Up to now we have discussed the properties of the optimal controller given a fixed $\mu\ge0$. In what follows, we show how to  compute an optimal multiplier $\mu^*$ based on Theorems~\ref{thm:KKT},~\ref{thm:OPTIMAL}. For any fixed $\mu\ge0$, we provide a closed-form expression for evaluating the risk functional $J_R(u^*(\mu))$. Moreover, we show that $J_R(u^*(\cdot))$ is a continuous function of $\mu$. Hence, if Slater's condition is satisfied, then based on Theorem~\ref{thm:OPTIMAL}, we can find the optimal multiplier $\mu^*$ by trivially applying bisection on $\mu$. 

The evaluation of the risk constraint functional $J_R(u^*(\mu))$ may be performed in a recursive fashion, as the following result suggests.
\begin{proposition}[\textbf{Risk Functional Evaluation}]\label{prop:LQR_Risk_Evaluation}  Assume that system~\eqref{FOR_EQN_system} is fully-observed: $y_k=x_k$ for all $k\ge 0$. Let Assumption \ref{FOR_ASS_noise} be in effect, and adopt the notation of Lemma~\ref{lem:Lagrangian_Dynamic_Programming}.
For fixed $\mu\ge0$, consider the control policy $u^*(\mu)$, as defined in~\eqref{eq:LQR_Optimal_Input}.
With terminal values $P_N=4QWQ$, $\zeta_N=m_{3}$, $d_N=0$, consider the backward recursions
\begin{align}
&P_{t-1}=(A+BK_{t-1})'P_t(A+BK_{t-1})+4QWQ,\nonumber\\
&\zeta_{t-1}=(A+BK_{t-1})'\zeta_t+m_3\nonumber\\
&\,\,\,+(A+BK_{t-1})'P_t\paren{Bl_{t-1}+\bar{w}}\,\,\,\, \textrm{and}\nonumber\\
&d_{t-1}=d_t+\Tr\big([P_{t-1}-4QWQ]W\big)\nonumber\\
&\,\,\,+2\zeta_t'(\bar{w}+Bl_{t-1})+(Bl_{t-1}+\bar{w})'P_{t}(Bl_{t-1}+\bar{w}).\nonumber
\end{align}
Then, the risk constraint in problem \eqref{eq:LQ_constrained_reformulated} may be evaluated by
\begin{align*}
J_R(u^*(\mu))&=x'_0(P_0-4QWQ)x_0+2(\zeta'_0-m_3)x_0\\
&+d_0-\Tr\big([P_{0}-4QWQ]W\big)
\end{align*}
Moreover, the risk-functional $J_R(u^*(\cdot))$ is continuous.
\end{proposition}
\begin{proof}
Omitted; it is similar to that of Proposition~\ref{prop:LQ_Risk_Evaluation} in Section~\ref{Section_Partially}. To prove continuity it is sufficient to invoke invertibility of $R$. We don't need Assumption~\ref{ass:moments_partially_observed} since the expressions above are deterministic. As a result, Assumption~\ref{FOR_ASS_noise} suffices. 
\end{proof}

Now, we can obtain the optimal solution $u^*(\mu^*)$ to the original problem~\eqref{eq:LQ_constrained_reformulated} for fully observed systems.

\begin{theorem}[\textbf{Primal-Optimal Solution}]\label{thm:optimal_LQR}
Assume that system~\eqref{FOR_EQN_system} is fully-observed: $y_k=x_k$ for all $k\ge 0$. Let Assumption \ref{FOR_ASS_noise} be in effect, and adopt the notation of Lemma~\ref{lem:Lagrangian_Dynamic_Programming}.
 Define the minimum feasible $\bar{\epsilon}_{\inf}$:
\[
\bar{\epsilon}_{\inf}\triangleq \inf_{u\in\mathcal{U}_0} J_R(u).
\]
Then, for any $\bar{\epsilon}>\bar{\epsilon}_{\inf}$, problem~\eqref{eq:LQ_constrained_reformulated} is feasible and the optimal solution is given by $u^*(\mu^*)$ based on~\eqref{eq:Optimal_Multiplier},~\eqref{eq:LQR_Optimal_Input}.
\end{theorem}
\begin{proof}
Omitted; it is similar to the proof of Theorem~\ref{thm:optimal_LQ}.
\end{proof}
Note that solving the problem
\[
\inf_{u\in\mathcal{U}_0} J_R(u)=\inf_{u\in\mathcal{U}_0} \E \sum_{k=1}^{N} \hat{x}'_kQWQ\hat{x}_k+2m_3'\hat{x}_k
\]
corresponds to finding a maximally risk-aware policy. Since the risk functional is quadratic, we can solve the problem following standard LQR theory. Note that the problem is singular since there is no input penalty in $J_R$ and the matrix $QWQ$ could be singular. Hence, there might be multiple optimal solutions. Nonetheless, we can still obtain an admissible optimal solution so that the infimum becomes minimum; such a solution will involve pseudo-inverses instead of inverses. More information about singular LQR control can be found in~\cite{lewis1981generalized}.

The above problem characterizes the minimum value of $\epsilon$ such that~\eqref{eq:LQ_constrained} is feasible. Note that as we increase $\epsilon$, we relax the risk-aware requirements. Let $\epsilon_{LQR}\triangleq J_R(u^*(0))$ be the value of the risk functional evaluated at the risk-neutral LQR optimal controller. Then, trivially if $\epsilon$ reaches a large value, i.e. larger than $\epsilon_{LQR}$, then the risk-neutral LQR controller will become feasible. After that point, if we keep increasing $\epsilon$, the risk-neutral LQR controller will always be the optimal solution to~\eqref{eq:LQ_constrained} with $\mu^*=0$. Hence, to obtain risk-aware behaviors, we need to select $\epsilon<\epsilon_{LQR}$.
\section{Optimal Risk-Aware LQ Control}\label{Section_Partially}
In this section we study problem~\eqref{eq:LQ_constrained_reformulated} in its full generality, when we only have access to partial state measurements. 
Fix a Lagrange multiplier $\mu\ge 0$ and recall the definition of Lagrangian $\LAG$ in~\eqref{eq:Lagrangian}.
Before we derive the optimal control law, let us simplify the form of the Lagrangian $\LAG$.
For brevity, denote the information up to time $t$ (extended state) by $z_t=(y_{0:t},u_{0:t-1})$, $z_0=y_0$.  Then, we get the following result.
\begin{lemma}[\textbf{Lagrangian Reformulation}]\label{lem:Lagrangian_partially_reformulated}
Let Assumption~\ref{FOR_ASS_noise} be in effect. Consider the sequential variational QCQP problem~\eqref{eq:LQ_constrained_reformulated} and define the inflated state penalty matrix
\begin{equation}
    Q_{\mu,t} \triangleq Q+4\mu QW_{t}Q.\nonumber
\end{equation}
Then, for every $u_t\in\LL_4(\F_t)$, $t\le N-1$, the Lagrangian function $\LAG$ can be expressed as
	\begin{equation}\label{eq:Lagrangian_partially_reformulated}
 \LAG(u,\mu)=\E \set{\sum_{t=0}^{N-1} g_{t}(z_t,u_t,\mu)}+g(\mu),
	\end{equation}
	where 
\begin{align*}
    g_{t}(z_t,u_t,\mu)=&\hat{x}'_{t+1}Q_{\mu,t}\hat{x}_{t+1}+2\mu\hat{x}'_{t+1}m_{3,t}+u'_{t}Ru_{t}\\
    g(\mu)=&-\mu\bar{\epsilon}+\E\sum_{t=0}^{N-1}\Tr(QW_t)+x'_0Qx_0.
\end{align*}
\end{lemma}
\begin{proof}
It follows from Proposition~\ref{prop:partially_existence} and 
\begin{align*}
    \E x'_tQx_t&=\E \set{\E (x'_tQx_t|\F_{t-1})}\\
    &=\E \hat{x}'_{t}Q\hat{x}_{t}+\E \set{\E (\delta'_tQ\delta_t|\F_{t-1})}\\
    &=\E \hat{x}'_{t}Q\hat{x}_{t}+\E\Tr(QW_{t-1}).\qedhere
\end{align*}
\end{proof}
Following the same convention as in the fully-observed case,
we define the \textit{optimal cost-to-go} at stage $k\le N-1$ as
\begin{equation*}
\LAG^*_k(z_{k},\mu)\triangleq \inf_{u_{k:N-1}\in \mathcal{U}_k} \E\set{\sum_{t=k}^{N-1}g_{t}(z_t,u_t,\mu)\Bigg|\F_{k}}.
\end{equation*}
where we omit the constant components of the Lagrangian. 

We may now derive a closed-form solution to \eqref{eq:dual}, which provides optimal risk-aware control policies for every fixed multiplier $\mu\ge0$. 
The above formulation in Lemma~\ref{lem:Lagrangian_partially_reformulated} only requires the noise signals to satisfy Assumption~\ref{FOR_ASS_noise}. However, to guarantee that our closed-form optimal controller below is well-defined and has finite fourth moments, we need the following sufficient stricter assumption, which strengthens Assumption~\ref{FOR_ASS_noise}.
\begin{assumption}[\textbf{Strict noise regularity}]\label{ass:moments_partially_observed}
The process noise $w_t$ has finite moments $\E \snorm{w_t}_2^p<\infty$, for any order $p\ge 1$ and any time $t\in\mathbb{N}$. 
\end{assumption}
\begin{theorem}[\textbf{LQ Risk-Aware Controllers}]\label{thm:partially_observed}
Let Assumption~\ref{ass:moments_partially_observed} be in effect. Fix a multiplier $\mu\ge 0$ and adopt the notation of Lemma~\ref{lem:Lagrangian_partially_reformulated}. Recall the definition of the refinement error $e_t$ in~\eqref{eq:estimation_error}. Then, the optimal cost-to-go at time $t$ is given by:
\begin{align}\label{eq:optimal_cost_to_go}
    \LAG^*_t(z_t,\mu)&=\hat{x}'_{t|t}P_t\hat{x}_{t|t}+2\zeta'_t\hat{x}_{t|t}+c_t,
\end{align}
where $P_t,\zeta_t,c_t$ are $\F_{t}$-measurable, given by the recursions:
\begin{align}
    V_t&=\E(P_{t+1}|\F_t)+Q_{\mu,t} \label{eq:aux_cost_to_go_V}\\
    \xi_{t}&=\E(P_{t+1}e_{t+1}+\zeta_{t+1}|\F_t)+\mu m_{3,t}\label{eq:aux_cost_to_go_xi}\\
    d_t&=\E\left\{c_{t+1}+e_{t+1}'P_{t+1}e_{t+1}+2e_{t+1}'\zeta_{t+1}|\F_t\right\}\label{eq:aux_cost_to_go_d}\\
    K_{t}&=-(B'V_tB+R)^{-1}B'V_{t}A\\
    l_t&=-(B'V_tB+R)^{-1}B'(\xi_t+V_t\bar{w})\\
    P_{t}&=(A+BK_{t})'V_t(A+BK_t)+K'_{t}RK_{t}\label{eq:aux_cost_P}\\
    \zeta_t&=(A+BK_t)'(\xi_t+V_t \bar{w}) \\
    c_t&=d_t+\bar{w}'V_t \bar{w}-l'_t(B'V_tB+R)l_t+2\xi_t'\bar{w}
\end{align}
with initial values $P_N=0,\zeta_N=0,c_N=0$. 
Additionally, an optimal control policy 
that achieves the dual value in \eqref{eq:FDUAL} may be expressed as:
\begin{equation}\label{eq:LQ_Optimal_Input}
    u^*_t(\mu)=K_t\hat{x}_{t|t}+l_t\in \LL_4(\F_{t}).
\end{equation}
\end{theorem}
The proof can be found in the Appendix.
Note that for the well-posedness of the solution, we only need i) the conditional moments in~\eqref{eq:aux_cost_to_go_V}-\eqref{eq:aux_cost_to_go_d} to be well-defined, and ii) fourth-moment integrability of the optimal inputs in~\eqref{eq:LQ_Optimal_Input}.
Assumption~\ref{ass:moments_partially_observed} is only a sufficient condition so that the above conditions are satisfied. 
It might not be a necessary condition. For example, in the fully-observed case  (Section~\ref{Section_Fully_Observed}), the milder Assumption~\ref{FOR_ASS_noise} suffices. Here, the technical difficulty stems from the fact that in the general case, $K_t$ is stochastic and potentially unbounded. 
We leave further discussion for future work.

Interestingly, in the partially observed case the control law is still affine-like. However, the linear and affine terms are no longer constants. They evolve based on a new filtered version of the Riccati difference equation, see~\eqref{eq:aux_cost_to_go_V},~\eqref{eq:aux_cost_P}. They are random variables that depend exclusively on the stochastic dynamics (noises) of the system.  The intuition behind the linear gain and the affine term is similar to the fully-observed case, however, there is a major difference. Instead of accounting only for the process noise $w_k$, we account for the whole prediction error $x_k-\hat{x}_k$, which captures also the estimation uncertainty.
The control policy and the estimation process are intertwined, in the sense that the latter affects the gain and the affine part. Hence, the control policy cannot be designed independently of the estimation process/noise statistics. In other words, the certainty equivalence property (see~\cite{Speyer2008STOCHASTIC}) does not hold. 
However, separation holds weakly in the sense that the optimal state estimator~$\hat{x}_{t|t}$ in~\eqref{eq:LQ_Optimal_Input} is the minimum mean-square error (mmse) estimator and can be designed separately from the optimal controller.

\subsection{Recovery of Primal-Optimal Solutions}
In this subsection, we provide a closed-form expression to evaluate the risk functional $J_R(u^*(\mu))$. Moreover, we show that $J_R(u^*(\cdot))$ is a continuous function of $\mu$. Similar to the fully-observed case, if Slater's condition is satisfied, then we can find the optimal multiplier $\mu^*$ by trivially applying bisection.
\begin{proposition}[\textbf{Risk Functional Evaluation}]\label{prop:LQ_Risk_Evaluation} Let Assumption \ref{ass:moments_partially_observed} be in effect, and adopt the notation of Lemma~\ref{lem:Lagrangian_partially_reformulated}. Recall the definition of the refinement error $e_t$ in~\eqref{eq:estimation_error}.
For fixed $\mu\ge0$, consider the control policy $u^*(\mu)$, as defined in~\eqref{eq:LQ_Optimal_Input}.
With terminal values $H_{N-1}=4QW_{N-1}Q$, $f_{N-1}=m_{3,N-1}$, $g_N=0$, consider the backward recursions
\begin{align*}
&\Theta_{t}=(A+BK_{t})'H_{t}(A+BK_{t})\\
&H_{t-1}=\E(\Theta_{t}|\F_{t-1})+4QW_{t-1}Q\nonumber\\
&\eta_t=(A+BK_{t})'(f_t+H_t(Bl_t+\bar{w}))\\
&f_{t-1}=\E(\eta_t+\Theta_{t}e_{t}|\F_{t-1})+m_{3,t-1}\nonumber\\
&\gamma_t=g_t+(Bl_t+\bar{w})'H_t(Bl_t+\bar{w})+2(Bl_t+\bar{w})'f_t\\
&g_{t-1}=\E(\gamma_t+e_t'\Theta_{t}e_{t}+2e'_{t}\eta_t|\F_{t-1})
\end{align*}
Then, the risk constraint in problem \eqref{eq:LQ_constrained_reformulated} may be evaluated by
\begin{equation}
J_R(u^*(\mu))=x'_0\E(\Theta_0)x_0+2\E(\eta'_0) x_0+\E(\gamma_0).\nonumber
\end{equation}
Moreover, the risk functional $J_R(u^*(\mu))$ is a continuous function of $\mu$.
\end{proposition}

Finally, for completeness we state the following theorem which completely characterizes the solution to the primal problem~\eqref{eq:LQ_constrained_reformulated}.
\begin{theorem}[\textbf{Primal-Optimal Solution}]\label{thm:optimal_LQ}
Let Assumption \ref{ass:moments_partially_observed} be in effect, and adopt the notation of Lemma~\ref{lem:Lagrangian_partially_reformulated}.
 Define the minimum feasible $\bar{\epsilon}_{\inf}$:
\[
\bar{\epsilon}_{\inf}\triangleq \inf_{u\in\mathcal{U}_0} J_R(u).
\]
Then, for any $\bar{\epsilon}>\bar{\epsilon}_{\inf}$, problem~\eqref{eq:LQ_constrained_reformulated} is feasible and the optimal solution is given by $u^*(\mu^*)$ based on~\eqref{eq:Optimal_Multiplier},~\eqref{eq:LQ_Optimal_Input}.
\end{theorem}
\begin{proof}
It follows from Proposition~\ref{prop:LQ_Risk_Evaluation} that $J_R(u^*(\cdot))$ is continuous. Hence, if Slater's condition is satisfied, optimality follows from Theorems~\ref{thm:KKT},~\ref{thm:OPTIMAL}. We only need to show that Slater's condition is satisfied under $\bar{\epsilon}>\bar{\epsilon}_{\inf}$. This follows from the definition of infimum; there exists $u^{\dagger}\in\mathcal{U}_0$ such that $J_R(u^{\dagger})\le \bar{\epsilon}_{\inf}+(\bar{\epsilon}-\bar{\epsilon}_{\inf})/2<\bar{\epsilon}$.
\end{proof}

Although the recursions of Theorem~\ref{thm:partially_observed} and Proposition~\ref{prop:LQ_Risk_Evaluation} provide a closed-form solution, they require knowledge of several conditional moments. In reality, these conditional moments might be hard to track. 
Another side effect of the stochasticity of $K_t$ is that the result of Corollary~\ref{cor:stability} will no longer apply for general partially-observed systems; the gains will not converge pointwise in general. As a consequence, any stability analysis of~\eqref{eq:LQ_Optimal_Input} might require further assumptions, beyond the scope of this paper. Nonetheless, in the case of Gaussian noise we can solve both challenges; we can compute these moments exactly and we can prove stability under certain controllability/observability conditions.

\section{Optimal Risk-Aware LQG Control}\label{Section_Gaussian}
In the special case of Gaussian measurement and process noise, the innovation error $\delta_t$ is actually independent of the past $\F_{t-1}$. Therefore, the moments defined in Proposition~\ref{prop:partially_existence} are deterministic, and the recursive formulas for the control policy and the risk-evaluation can be simplified dramatically. In this section, we focus on exactly this case and provide explicit formulas for optimal risk-aware LQG controllers.
\begin{assumption}[\textbf{Gaussian Noise}]\label{ass:gaussian}
The process noise and measurement noise $w_k,\,v_k$ are jointly i.i.d. Gaussian with mean $\bar{w},\,0$ respectively and covariance
\[
\E\matr{{c}w_k\\v_k}\matr{{c}w_k\\v_k}'=\matr{{cc}W&0\\0&S}.
\]
\end{assumption}
\vspace{1bp}
\begin{theorem}[\textbf{LQG Risk-Aware Controllers}]\label{thm:gaussian}
Let Assumption~\ref{ass:gaussian} be in effect. Fix a multiplier $\mu\ge 0$ and consider the notation of Theorem~\ref{thm:partially_observed}. The innovation sequence $\delta_t$, $t\ge 0$ is Gaussian and independent with covariance given by the forward recursion (Kalman Filter)
\begin{align*}
    W_{t+1}&=AW_tA'+W-AW_tC'(CW_tC'+S)^{-1}CW_tA'
\end{align*}
for $W_0=0$. The mean square estimate $\hat{x}_{t|t}$ is given by
\[
\hat{x}_{t|t}=\hat{x}_{t}+W_tC'(CW_tC'+S)^{-1}(y_t-\hat{x}_{t}).
\]
Consider the backward recursion
\begin{align}
    V_{t-1}&=(A+BK_{t})'V_t(A+BK_t)+K'_{t}RK_{t}+Q_{\mu,t-1}\\
    \xi_{t-1}&=(A+BK_t)'(\xi_t+V_t\bar{w})\\
    K_{t}&=-(B'V_tB+R)^{-1}B'V_{t}A\\
    l_t&=-(B'V_tB+R)^{-1}B'\xi_t
\end{align}
with initial values $V_{N-1}=Q_{\mu,N-1}$, $\xi_{N-1}=0$. 
An optimal control policy 
that achieves the dual value in \eqref{eq:FDUAL} may be expressed as
\begin{equation}\label{eq:LQG_optimal_input}
    u^*_t(\mu)=K_t\hat{x}_{t|t}+l_t.
\end{equation}
\end{theorem}

\begin{proof}
The properties of $\delta_t$ and the recursions for $W_t$, $\hat{x}_{t|t}$ follow from standard Kalman Filter theory~\cite{anderson2005optimal}. We also have $m_{3,t}=0$ since the variables $\delta_t$ are Gaussian and, thus, symmetric. Since both $W_t,\,m_{3,t}$ are deterministic, matrices $V_t,\,P_t,\,\xi_t,\,\zeta_t$ in the statement of Theorem~\ref{thm:partially_observed} are also deterministic. Hence, we have $\E(P_{t}e_t|\F_{t})=0$ and we can remove the conditional expectations from $V_t,\xi_t$ as
\[
V_t=P_{t+1}+Q_{\mu,t},\, \xi_t=\zeta_{t+1},
\]
and the result now follows from Theorem~\ref{thm:partially_observed}.
\end{proof}
Unlike the fully-observed case, the inflated matrix takes into account not only the instantaneous process error $W$ but the whole prediction error. In other words, we account also for the uncertainty in the prediction due to partial observability.

The aforementioned property also differentiates our controller from classical risk-neutral Linear Quadratic Gaussian (LQG) control. The state penalties  $Q_{\mu,t}$ are inflated, time-varying, and they depend on the filtering process $W_t$ itself, whereas, in risk-neutral LQG, the control design is completely independent of the noise statistics. Hence, we obtain a novel family of risk-aware LQG policies, which include the classical LQG ($\mu=0$) as a special case. 

In the case of Gaussian noise the third moment $m_{3,t}=0$ is always zero due to symmetry. Contrary to the non-Guassian noise case, the affine term only accounts for the mean value of the noise. As a result, risk-aware behavior is mainly imposed through the gain $K_t$.

Finally, we prove a stability result for the closed-loop system under certain observability conditions.
\begin{assumption}[\textbf{Observability}]\label{ass:observability}
The pair $(A,C)$ is detectable, the pair $(A,W^{1/2})$ is stabilizable, and the covariance of the measurement noise is strictly positive definite $S \succ 0$.
\end{assumption}
To prove stability let us assume that  we start estimating/controlling the system at some arbitrary time $t_0<N$ instead of $0$, with $x_{t_0}$ deterministic and known. Based on this, all recursions in the statement of Theorem~\ref{thm:gaussian} are extended to hold for any $t=t_0,\dots,N$, with $W_{t_0}=0,\,\hat{x}_{t_0|t_0}=x_{t_0}$. We will prove that stability is achieved as we let the initial state $t_0$ and the horizon $N$ go to $-\infty$ and $+\infty$ respectively.

To simplify the proof, we also assume that the state penalty is strictly positive definite $Q\succ 0$. The proof can be extended to the case $Q\succeq 0$ at the cost of more complicated arguments, but we omit it in this paper-see proof in the Appendix for discussion. 
\begin{theorem}[\textbf{Stability}]\label{thm:gaussian_stability}
Consider the forward and backward recursions of Theorem~\ref{thm:gaussian} extended to the interval $t\in \set{t_0,\dots,N-1}$, with initial conditions $W_{t_0}=0$, $\hat{x}_{t_0|t_0}=0$ and terminal conditions as is. Let Assumptions~\ref{ass:controllability},~\ref{ass:gaussian},~\ref{ass:observability} be in effect. Let $V$, $W_{\infty}$ be the stabilizing solutions to the Discrete Algebraic Riccati Equations (DARE):
\begin{align*}
    W_{\infty}&=AW_{\infty}A'+W-AW_{\infty}C'(CW_{\infty}C+S)^{-1}CW_{\infty}A'\\
    V&=A'VA+Q_{\mu,\infty}-A'VB(B'VB+R)^{-1}B'VA,
\end{align*}
with $Q_{\mu,\infty}=Q+\mu QW_{\infty}Q$. Let $K\triangleq -(B'VB+R)^{-1}B'VA$ be the respective control gain, with $\bar{A}\triangleq A+BK$ the closed-loop matrix. Then, the closed-loop matrix $\bar{A}$ is stable and
\begin{equation}\label{eq:lqg_finite_time_error}
    \snorm{V_t-V}_2 \le \mathcal{C}_1\snorm{\bar{A}^{N-t-1}}+\mathcal{C}_2\snorm{W_t-W_{\infty}}_2,
\end{equation}
where $\mathcal{C}_1,\mathcal{C}_2$ are some positive constants that depend on the system parameters and $V$.
As a result, if we let $N\rightarrow \infty$, $t_0\rightarrow -\infty$:
\begin{align*}
 V_t&\rightarrow V,\quad K_t\rightarrow K,\quad W_t\rightarrow W\\
 \xi_t&\rightarrow \xi\triangleq (I-(A+BK)')^{-1}(A+BK)'V\bar{w}\\
 l_t&\rightarrow -(B'VB+R)^{-1}B'\xi,
\end{align*}
exponentially fast.
\end{theorem}
The intuition behind sending $t_0$ and $N$ to $-\infty$, $\infty$ is the following. This corresponds to a doubly infinite horizon problem, where the estimation process has started infinitely long ago in the past and the control process is running continuously for an infinite amount of time. 
 At first sight, the result seems to be equivalent to proving stability of the classical risk-neutral LQG controller. However, by a more careful examination, equation~\eqref{eq:lqg_finite_time_error} is different from 
classical LQG. The reason is that the second term $\mathcal{C}_2\snorm{W_t-W_{\infty}}_2$ shows up in the error only in the case of risk-aware LQG. While estimation and control are designed independently in risk-neutral LQG, in the case of risk-aware LQG, the estimation procedure affects the convergence of the controller to its steady-state.

Similar to LEQG control for partially observed systems, our risk-aware controller regulates the state more strictly. However, this is achieved via a different mechanism, that is, via the inflation of the state penalty term. As a result, in our formulation, stability is guaranteed for any choice of $\mu$ in~\eqref{eq:LQG_optimal_input}. This is unlike LEQG control, which might be unstable if we do not tune the exponential parameter $\theta$ carefully~\cite{Speyer2008STOCHASTIC}. Notice also that the optimal estimator for our risk-aware LQG controller in~\eqref{eq:LQG_optimal_input} is the minimum mean-square estimator. This is different from LEQG control, where the estimator is a biased version of the minimum mean-square estimator. 

\section{Simulations And Discussion}\label{Section_Simulations}
\begin{figure}[t]
	\centering
	\includegraphics[width=0.8\columnwidth]{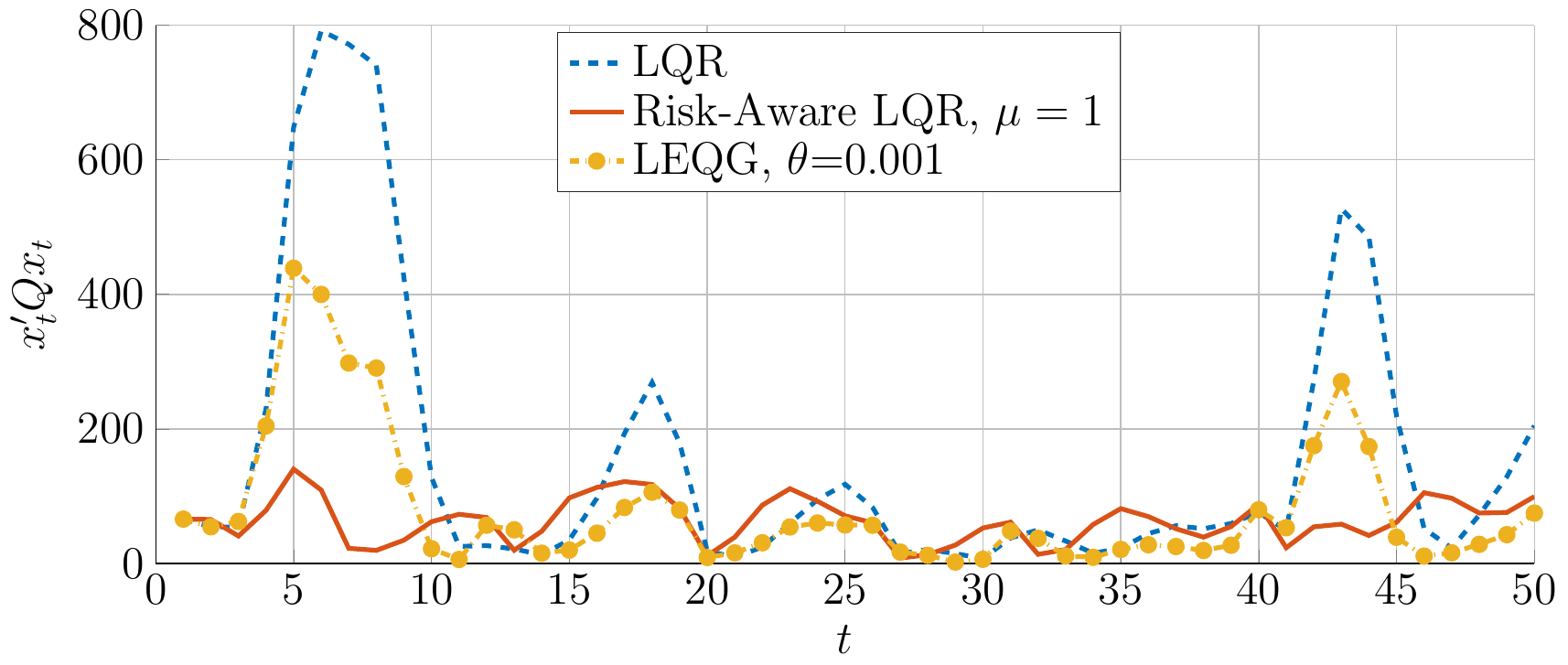}
	\caption{Evolution of the state penalties $x'_kQx_k$, over the first $50$ steps. Notice that our risk-aware LQR controller indeed limits the variability of $x'_kQx_k$. In fact, it sacrifices performance under small wind forces, but protects the system against large wind gusts, for example at time $5-10$.}
	\label{fig:xQx}
\end{figure}
Consider a flying robot that moves on a horizontal plane, i.e., the Euclidean space $\mathbb{R}^2$. We assume that its linearized dynamics can be abstracted by a double integrator as
\begin{equation*}
    x_{k+1}=\matr{{cccc}1&T_s&0&0\\0&1&0&0\\0&0&1&T_s\\0&0&0&1}x_k+\matr{{cc}\tfrac{T_s^2}{2}&0\\T_s&0\\0&\tfrac{T_s^2}{2}\\0&T_s}(\eta_k+d_k),
\end{equation*}
where  $T_s=0.5$ is the sampling time, $x_{k,1}$, $x_{k,3}$ are the position coordinates, $x_{k,2}$, $x_{k,4}$ the respective velocities and $\eta_k$ is the acceleration input. Let $d_k$ be a wind disturbance force that acts on the robot, which is modeled as follows: We assume that $d_{k,1}$ constitutes the dominant wind direction with non-zero mean and large variability, while the orthogonal direction $d_{k,2}$ is a weak wind direction with zero mean and small variability.
We model $d_{k,1}$ as a mixture of two gaussians $\mathcal{N}(30,30)$, $\mathcal{N}(80,60)$ with weights $0.8$ and $0.2$, respectively. This bimodal distribution models the presence of infrequent but large wind gusts.  The weak direction $d_{k,2}$ is modeled as zero-mean Gaussian~$\mathcal{N}(0,5)$.
If we cancel the mean of $d_k$ by applying $\eta_k=u_k-\E d_k$, then the system can be re-written in terms of~\eqref{FOR_EQN_system}, where $w_k=B(d_k-\E d_k)$ is now a zero-mean disturbance $\bar{w}=0$, and $u_k$ is the exogenous input. 

Consider now the LQR problem with parameters 
\[
Q = \textrm{diag}(1,0.1,2,0.1)
\Neg{2}\quad\textrm{and}\quad R=I,
\] 
and a horizon of length $N=5000$. We primarily compare our risk-aware LQR formulation with the classical, risk-neutral LQR via simulations. To tune our controller, we vary $\mu$ in~\eqref{eq:LQR_Optimal_Input} directly instead of varying $\epsilon$. We also (heuristically) compare our controller with the exponential (LEQG) method, even though the noise is not Gaussian, by plugging in the second order statistics $W$. Let the tuning parameter of LEQG be $\theta$. Note that the exponential problem is well defined only if $\theta<0.001276$ (roughly), where the ``neurotic breakdown" occurs~\cite{Whittle1981}. For the purpose of comparison, we simulate all schemes under the same noise sequence $w_{0:N}$.

In Fig.~\ref{fig:xQx}, we see the evolution of the state penalty terms $x_k'Qx_k$, for the first $50$ time steps, under the different control schemes.  By slightly sacrificing performance under small wind forces, our risk-aware LQR controller forces the state to have less variability and  protects the robot against large gusts. On the other hand, the state penalty can grow very large under the risk-neutral and LEQG schemes. This behavior is illustrated  more clearly in Fig.~\ref{fig:cdf}, where we present the time-empirical cumulative distribution of the state penalties for all $N$ time steps. The time-empirical "probability" of suffering large state penalties is drastically smaller compared to LQR or LEQG.

\begin{figure}[t]
	\centering
	\includegraphics[width=0.8\columnwidth]{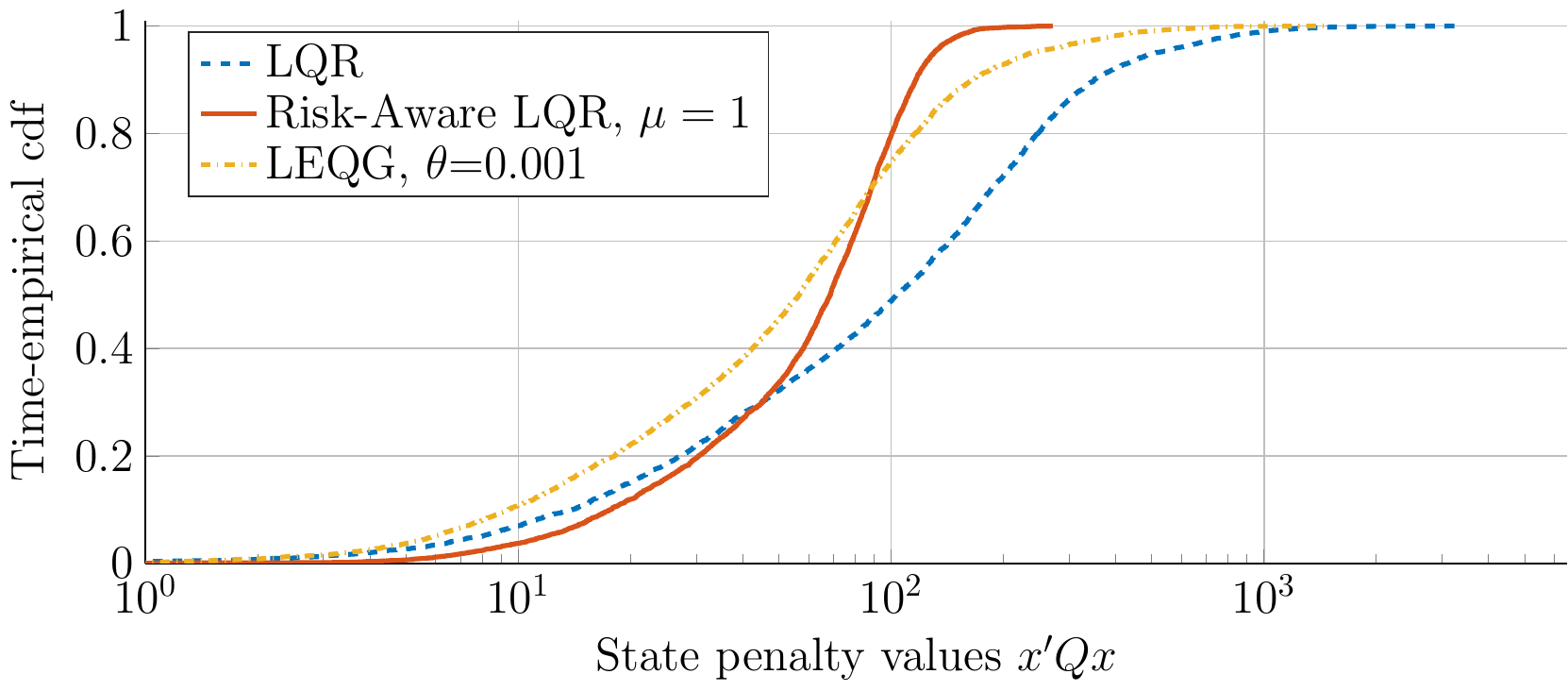}
	\caption{The time-empirical cdf for the state penalties $x_k'Qx_k$, $k\le N$, for the LQR (risk-neutral), our method, and LEQG. Our method sacrifices some average performance but exhibits much smaller variability for the state penalties. It protects the system against rare but large wind gusts.}
	\label{fig:cdf}
\end{figure}
\begin{figure}[t]
	\centering
	\includegraphics[width=0.8\columnwidth]{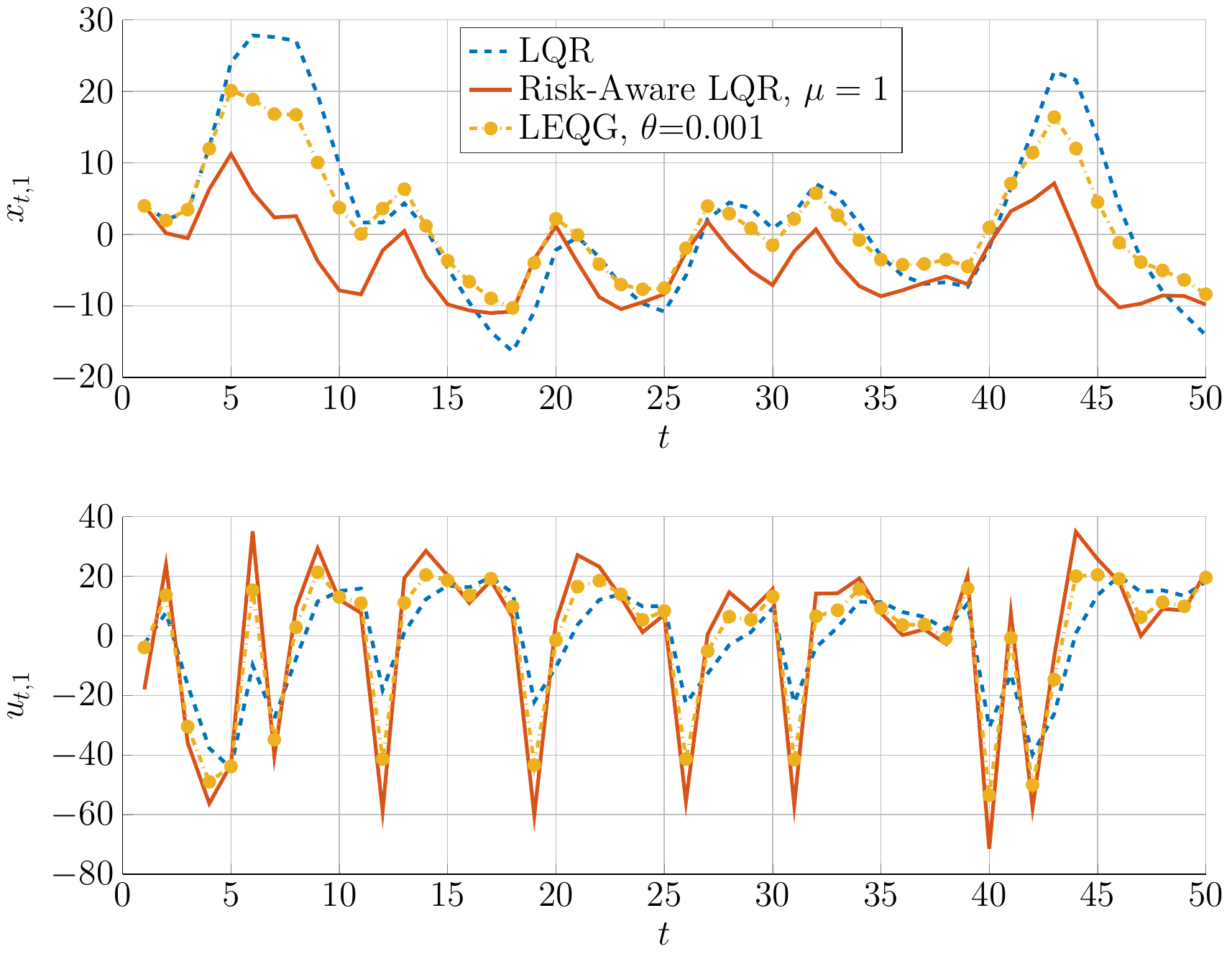}
	\caption{Evolution of the state $x_{k,1}$, and the input $u_{k,1}$ over the first $50$ steps. The controller pushes the state away from the direction of the large gusts, which helps the robot to avoid extreme perturbations. Meanwhile, by inflating the state penalty with the $\mu QWQ$, we force the state-feedback component to be more cautious with the state. Naturally, being more cautious with the state requires extra control effort.}
	\label{fig:x1}
\end{figure}
To better illustrate how the proposed risk-aware controller works, we also discuss the evolution of the position $x_{k,1}$ and the input $u_{k,1}$, as shown in Fig. \ref{fig:x1}, for the first $50$ steps. First, we observe that the controller pushes the state $x_{k,1}$ towards negative values, away from the direction of the large gusts.
Second, notice that we penalize $x_{k,3}$ more in $Q$. In fact, the risk-neutral LQR results in the steady state gains $K_{\mathrm{LQR},11}=-0.697$, $K_{\mathrm{LQR},12}=-1.201$, $K_{\mathrm{LQR},23}=-0.925$, $K_{\mathrm{LQR},24}=-1.376$, i.e., it is stricter with direction $x_{k,3}$. However, $x_{k,1}$ exhibits more variability due to the strong wind direction. In contrast, our risk-aware scheme adapts to the noise in a principled way. Due to the inflation term $\mu QWQ$, our scheme returns the steady-state gains $K_{11}=-2.1008$, $K_{12}=-2.2132$, $K_{23}=-1.1161$, $K_{24}=-1.5131$, which means that the risky direction $x_{k,1}$ is controlled more strictly. Naturally, being more cautious with the state leads to higher control effort, as shown in Fig.~\ref{fig:x1}. Lastly, although the LEQG controller is also more state-cautious, it is agnostic to the heavy tails of the wind distribution. Hence, it still suffers from large perturbation due to the wind gusts.

\subsection{Risk-aware LQG control}
In this section we evaluate the risk-aware LQG controller developed in Section~\ref{Section_Gaussian}.
We use the penalty matrices \[
Q = \textrm{diag}(1,0.5,2,0.5)
\Neg{2}\quad\textrm{and}\quad R=I.
\] 
However, the process noise is now mean-zero Gaussian, with $d_{k,1}\sim \mathcal{N}(0,30)$ and $d_{k,2}\sim \mathcal{N}(0,5)$.
For the measurement model, we assume
\[
C=\matr{{cccc}1&0&0&0\\0&0&1&0},\, \E{v_k v'_k}=\matr{{cc}5&2\\2&2},
\]which implies that we have access to position measurements. 
We compare our risk-aware LQG controller with the risk-neutral LQG and the LEQG schemes. For the LEQG scheme, we used the non-delayed version~\cite{Speyer2008STOCHASTIC}[Th 10.5].

We simulated the system for a horizon of length $N=3000$. The evolution of the state penalties for the first $50$ time steps is shown in Fig.~\ref{fig:partially_observed}. As expected from~\eqref{eq:LQG_optimal_input}, the controller regulates the state more strictly compared to the risk-neutral LQG controller, by inflating the $Q$ matrix. Note that contrary to the fully-observed example, the noise is zero-mean Gaussian here, hence, there is no affine term in the optimal controller. We observed that the LEQG controller has similar behavior for small values of the exponential parameter $\theta$. 

A more detailed comparison is shown in Fig.~\ref{fig:partially_observed_cdf}, where the time-empirical cumulative distributions of the state penalties and the input penalties over $3000$ time steps are shown. As we require our controller to be more risk-aware (we increase $\mu$), the state penalties become smaller since the risky directions of the state are regulated more strictly. Naturally, regulating the state more strictly requires more control effort, hence the input penalties become larger. As we approach the maximally-risk aware controller ($\mu=100$), we achieve the smallest state penalties but the largest input penalties.

\begin{figure}[t]
	\centering
	\includegraphics[width=0.8\columnwidth]{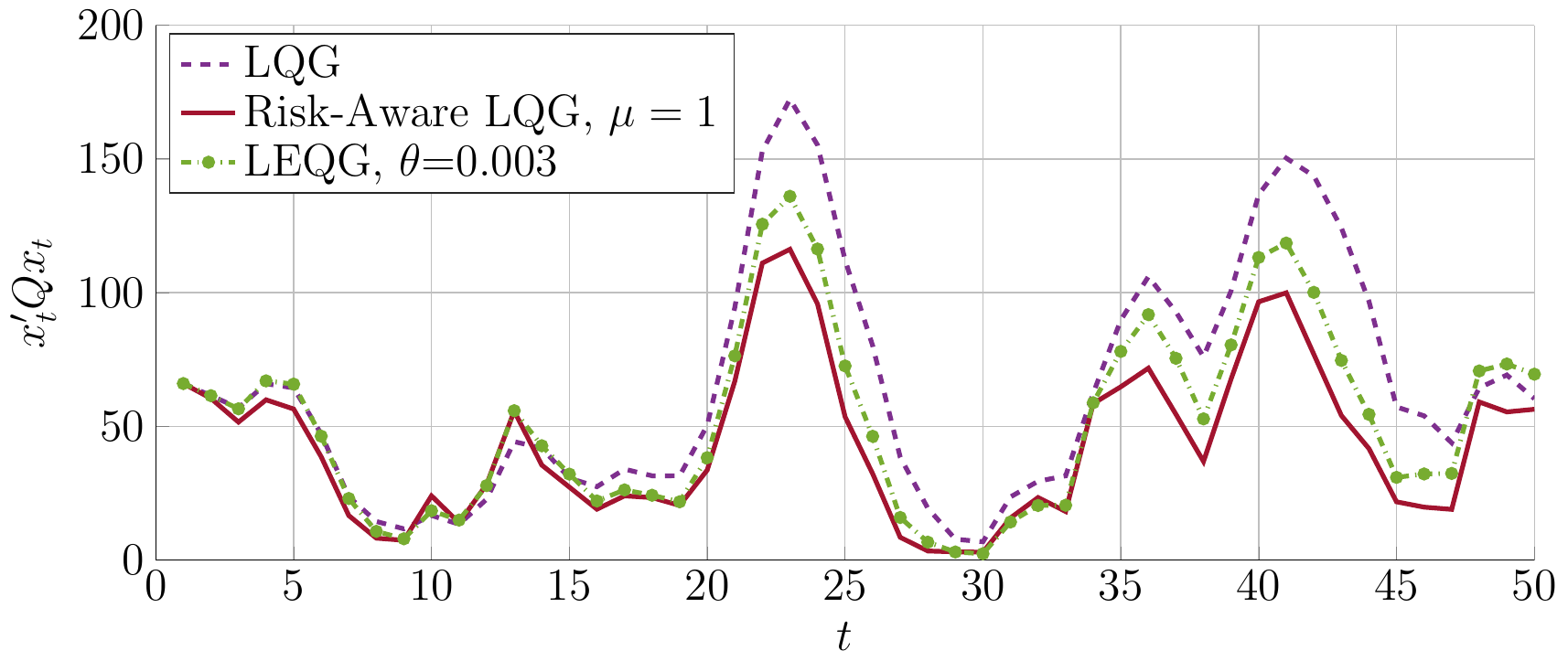}
	\caption{Evolution of the state penalties $x'_kQx_k$,
	over the first $50$ steps. 
		Our risk-aware LQG controller regulates the state more strictly by using more control effort. A similar property holds for the LEQG controller. The risk-neutral LQG controller suffers from larger state perturbations.
	}
	\label{fig:partially_observed}
\end{figure}

\begin{figure}[t]
	\centering
	\includegraphics[width=0.8\columnwidth]{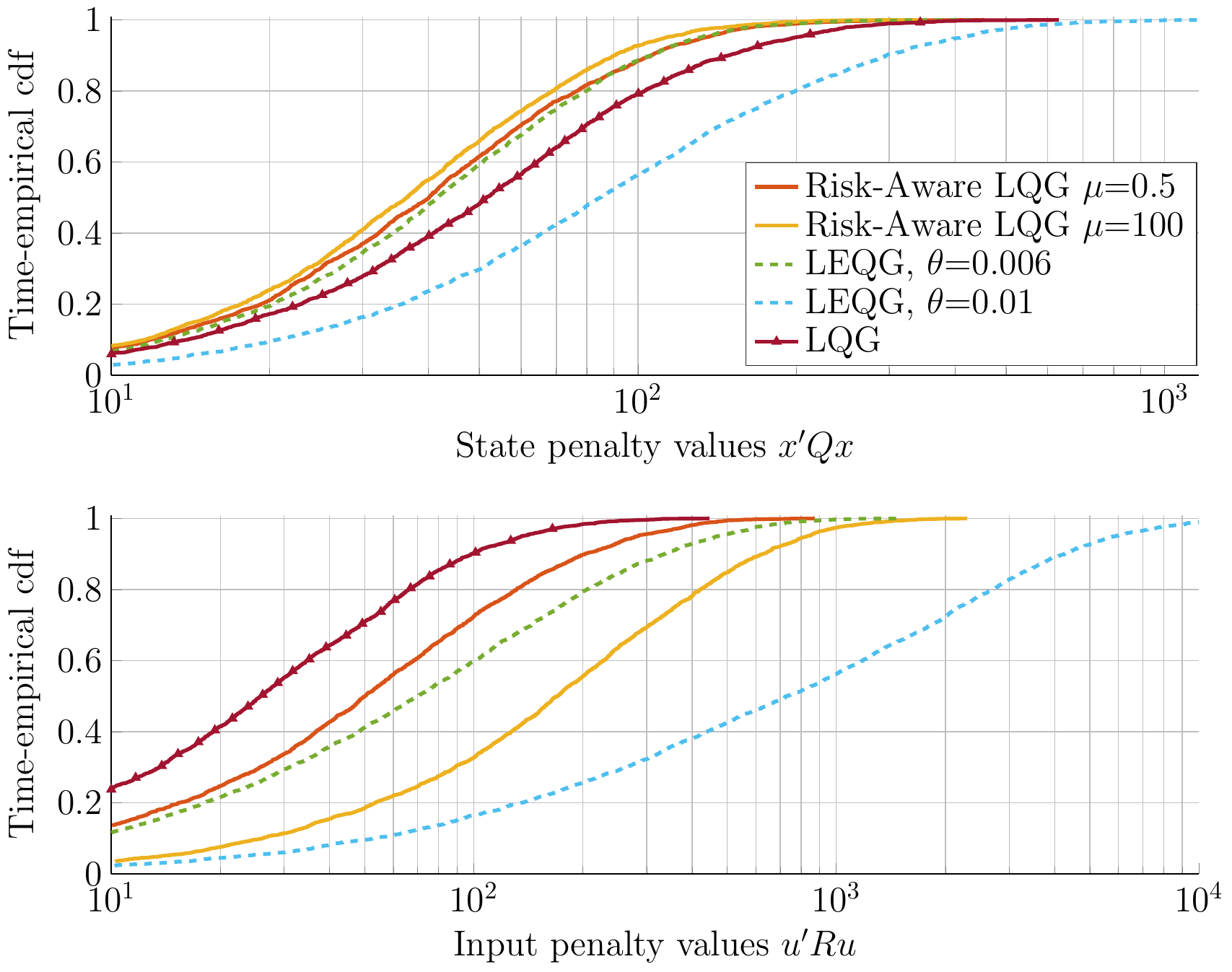}
	\caption{The time-empirical cdfs for the state and input penalties $x_k'Qx_k$, $u_k'Ru_k$ under our method, the LEQG controller, and the LQG (risk-neutral) controller. For better visibility we show the probabilities (y-axis) for penalties above the value of $10$ (x-axis). Our risk-aware formulation is more flexible, resulting in a wider variety of tradeoff curves between control effort and state regulation. It is more intuitive and easy to tune; as we increase $\mu$ we give more emphasis into regulating the risky directions of the state. The LEQG controller is less intuitive to tune. If $\theta$ increases too much there is a sharp decline in performance.
	}
	\vspace{-0.2cm}
	\label{fig:partially_observed_cdf}
\end{figure}

For small exponential parameters $\theta$ (below $0.006$) a similar behavior is observed in the case of the LEQG controller. As we increase $\theta$, the state is regulated more strictly at the expense of increased control effort. We achieve the smallest state penalties for roughly $\theta=0.006$. After this value, the tradeoff between control effort and state regulation becomes worse; for example, here both the state penalties and the input penalties increase as we increase $\theta$ past $0.006$. In fact, as $\theta$ approaches the ``neurotic breakdown" point, e.g. for $\theta=0.01$, both penalties become excessively large. This might be expected since the LEQG maximally risk-aware controller is very conservative, treating the noise as being adversarial rather than being stochastic, which is a different regime. On the contrary, our risk-aware LQG controller is well-behaved regardless the value of $\mu$. Hence it is more easy to tune and offers a wider variety of tradeoff curves between control effort and state regulation. For example, if we compare the risk-aware LQG controller for $\mu=0.5$ and the LEQG controller for $\theta=0.006$, then the risk-aware LQG controller achieves similar state penalties with less control effort.

\section{Conclusion}\label{sec:conclusion}
We studied a novel risk-aware formulation of the classical Linear Quadratic control problem, where we minimize average performance, subject to predictive variance constraints. 
This gives rise to risk-aware controllers which trade between average performance and protection 
against uncommon but strong random disturbances.
Our formulation is well-defined for general noise distributions, without requiring the existence of the respective moment generating functions. 
We characterized the optimal control laws for general partially-observed systems, which are affine with respect to the minimum mean-square state estimate.  We provided explicit risk-aware control formulas for the special cases of i) fully-observed systems and ii) Gaussian noise. The optimal controllers are easy to tune and are internally stable under standard controllability/observability conditions.

Moving forward, there are numerous interesting research directions. First, our formulation places more emphasis on regulating the state at the cost of increased control effort. To mitigate this, we could potentially include input power constraints~\cite{gattami2009generalized} in the quadratic formulation~\eqref{eq:LQ_constrained_reformulated}. Another open problem is explicitly computing the optimal control~\eqref{eq:LQ_Optimal_Input} in the case of partially-observed systems with non-Gaussian noise. Providing explicit closed-form expressions in this case is a hard problem, since it requires tracking of conditional moments. However, it might be possible to provide computational methods, which solve the problem approximately.
Lastly, our predictive
variance constraint is based on one-step-ahead prediction. In
some cases, this might make our controller
more myopic. Increasing the
prediction horizon, however, might not always preserve the quadratic
form of the constraint. In future work, we would also like to
address this issue.

\bibliographystyle{IEEEtran}
\bibliography{IEEEabrv,risk_literature,library_fixed}

\begin{thebibliography}{10}
\providecommand{\url}[1]{#1}
\csname url@samestyle\endcsname
\providecommand{\newblock}{\relax}
\providecommand{\bibinfo}[2]{#2}
\providecommand{\BIBentrySTDinterwordspacing}{\spaceskip=0pt\relax}
\providecommand{\BIBentryALTinterwordstretchfactor}{4}
\providecommand{\BIBentryALTinterwordspacing}{\spaceskip=\fontdimen2\font plus
\BIBentryALTinterwordstretchfactor\fontdimen3\font minus
  \fontdimen4\font\relax}
\providecommand{\BIBforeignlanguage}[2]{{%
\expandafter\ifx\csname l@#1\endcsname\relax
\typeout{** WARNING: IEEEtran.bst: No hyphenation pattern has been}%
\typeout{** loaded for the language `#1'. Using the pattern for}%
\typeout{** the default language instead.}%
\else
\language=\csname l@#1\endcsname
\fi
#2}}
\providecommand{\BIBdecl}{\relax}
\BIBdecl

\bibitem{Ahlen2019}
A.~Ahl{\'{e}}n, J.~Akerberg, M.~Eriksson, A.~L. F.~J. Isaksson, T.~Iwaki, K.~H.
  Johansson, S.~Knorn, T.~Lindh, and H.~Sandberg, ``Towards {W}ireless
  {C}ontrol in {I}ndustrial {P}rocess {A}utomation: {A} {C}ase {S}tudy at a
  {P}aper {M}ill,'' \emph{Control Systems, IEEE}, vol.~39, no.~5, pp. 36--57,
  2019.

\bibitem{Bruno2016}
S.~Bruno, S.~Ahmed, A.~Shapiro, and A.~Street, ``Risk-{N}eutral and
  {R}isk-{A}verse {A}pproaches to {M}ultistage {R}enewable {I}nvestment
  {P}lanning under {U}ncertainty,'' \emph{European Journal of Operational
  Research}, vol. 250, no.~3, pp. 979--989, May 2016.

\bibitem{Moazeni2015}
S.~Moazeni, W.~B. Powell, and A.~H. Hajimiragha, ``Mean-{C}onditional
  {V}alue-at-{R}isk {O}ptimal {E}nergy {S}torage {O}peration in the {P}resence
  of {T}ransaction {C}osts,'' \emph{IEEE Transactions on Power Systems},
  vol.~30, no.~3, pp. 1222--1232, May 2015.

\bibitem{Markowitz1952}
H.~Markowitz, ``Portfolio {S}election,'' \emph{The Journal of Finance}, vol.~7,
  no.~1, pp. 77--91, Mar. 1952.

\bibitem{Follmer2002}
H.~F{\"{o}}llmer and A.~Schied, ``Convex {M}easures of {R}isk and {T}rading
  {C}onstraints,'' \emph{Finance and Stochastics}, vol.~6, no.~4, pp. 429--447,
  Oct. 2002.

\bibitem{Shang2018}
D.~Shang, V.~Kuzmenko, and S.~Uryasev, ``Cash {F}low {M}atching with {R}isks
  {C}ontrolled by {B}uffered {P}robability of {E}xceedance and {C}onditional
  {V}alue-at-{R}isk,'' \emph{Annals of Operations Research}, vol. 260, no. 1-2,
  pp. 501--514, Jan. 2018.

\bibitem{Kim2019}
S.-K. Kim, R.~Thakker, and A.-A. Agha-Mohammadi, ``Bi-{D}irectional {V}alue
  {L}earning for {R}isk-{A}ware {P}lanning {U}nder {U}ncertainty,'' \emph{IEEE
  Robotics and Automation Letters}, vol.~4, no.~3, pp. 2493--2500, 2019.

\bibitem{Pereira2013}
A.~A. Pereira, J.~Binney, G.~A. Hollinger, and G.~S. Sukhatme, ``Risk-{A}ware
  {P}ath {P}lanning for {A}utonomous {U}nderwater {V}ehicles using {P}redictive
  {O}cean {M}odels,'' \emph{Journal of Field Robotics}, vol.~30, no.~5, pp.
  741--762, Sep. 2013.

\bibitem{Ma2018}
W.-J. Ma, C.~Oh, Y.~Liu, D.~Dentcheva, and M.~M. Zavlanos, ``Risk-{A}verse
  {A}ccess {P}oint {S}election in {W}ireless {C}ommunication {N}etworks,''
  \emph{IEEE Transactions on Control of Network Systems}, vol. 5870, no.~c, pp.
  1--1, 2018.

\bibitem{samuelson2018safety}
S.~{Samuelson} and I.~{Yang}, ``Safety-{A}ware {O}ptimal {C}ontrol of
  {S}tochastic {S}ystems {U}sing {C}onditional {V}alue-at-{R}isk,'' in
  \emph{2018 Annual American Control Conference (ACC)}, 2018, pp. 6285--6290.

\bibitem{chapman2019cvar}
M.~P. {Chapman}, J.~{Lacotte}, A.~{Tamar}, D.~{Lee}, K.~M. {Smith}, V.~{Cheng},
  J.~F. {Fisac}, S.~{Jha}, M.~{Pavone}, and C.~J. {Tomlin}, ``A
  {R}isk-{S}ensitive {F}inite-{T}ime {R}eachability {A}pproach for {S}afety of
  {S}tochastic {D}ynamic {S}ystems,'' in \emph{2019 American Control Conference
  (ACC)}, 2019, pp. 2958--2963.

\bibitem{jacobson1973optimal}
D.~{Jacobson}, ``Optimal {S}tochastic {L}inear {S}ystems with {E}xponential
  {P}erformance {C}riteria and their {R}elation to {D}eterministic
  {D}ifferential {G}ames,'' \emph{IEEE Transactions on Automatic Control},
  vol.~18, no.~2, pp. 124--131, 1973.

\bibitem{Whittle1981}
P.~Whittle, ``Risk-{S}ensitive {L}inear/{Q}uadratic/{G}aussian {C}ontrol,''
  \emph{Adv. Appl. Prob}, vol.~13, pp. 764--777, 1981.

\bibitem{bacsar2000risk}
T.~Ba{\c{s}}ar, ``Risk-{A}verse {D}esigns: {F}rom {E}xponential {C}ost to
  {S}tochastic {G}ames,'' in \emph{System Theory}.\hskip 1em plus 0.5em minus
  0.4em\relax Springer, 2000, pp. 131--143.

\bibitem{pham2012linear}
K.~D. Pham, \emph{Linear-Quadratic Controls in Risk-Averse Decision Making:
  Performance-Measure Statistics and Control Decision Optimization}.\hskip 1em
  plus 0.5em minus 0.4em\relax Springer Science \& Business Media, 2012.

\bibitem{roulet2019convergence}
V.~Roulet, M.~Fazel, S.~Srinivasa, and Z.~Harchaoui, ``On the {C}onvergence of
  the {I}terative {L}inear {E}xponential {Q}uadratic {G}aussian {A}lgorithm to
  {S}tationary {P}oints,'' 2019.

\bibitem{Speyer1992}
J.~Speyer, C.-H. Fan, and R.~Banavar, ``Optimal {S}tochastic {E}stimation with
  {E}xponential {C}ost {C}riteria,'' in \emph{Proceedings of the 31st IEEE
  Conference on Decision and Control}.\hskip 1em plus 0.5em minus 0.4em\relax
  Institute of Electrical and Electronics Engineers (IEEE), Aug. 1992, pp.
  2293--2298.

\bibitem{Dey1999}
S.~Dey and J.~B. Moore, ``Finite-{D}imensional {R}isk-{S}ensitive {F}ilters and
  {S}moothers for {D}iscrete-{T}ime {N}onlinear {S}ystems,'' \emph{IEEE
  Transactions on Automatic Control}, vol.~44, no.~6, pp. 1234--1239, 1999.

\bibitem{Moore1997}
J.~B. Moore, R.~J. Elliott, and S.~Dey, ``Risk-{S}ensitive {G}eneralizations of
  {M}inimum {V}ariance {E}stimation and {C}ontrol,'' \emph{Journal of
  Mathematical Systems, Estimation, and Control}, vol.~7, no.~1, pp. 123--126,
  1997.

\bibitem{Dey1997}
S.~Dey and J.~B. Moore, ``Risk-{S}ensitive {F}iltering and {S}moothing via
  {R}eference {P}robability {M}ethods,'' \emph{IEEE Transactions on Automatic
  Control}, vol.~42, no.~11, pp. 1587--1591, 1997.

\bibitem{Bauerle2014more}
N.~Bäuerle and U.~Rieder, ``More {R}isk-{S}ensitive {M}arkov {D}ecision
  {P}rocesses,'' \emph{Mathematics of Operations Research}, vol.~39, no.~1, pp.
  105--120, 2014.

\bibitem{Speyer2008STOCHASTIC}
J.~L. Speyer and W.~H. Chung, \emph{Stochastic Processes, Estimation, and
  Control}.\hskip 1em plus 0.5em minus 0.4em\relax Siam, 2008, vol.~17.

\bibitem{A.2018}
L.~A. Prashanth and M.~Fu, ``Risk-{S}ensitive {R}einforcement {L}earning: {A}
  {C}onstrained {O}ptimization {V}iewpoint,'' \emph{arXiv preprint,
  arXiv:1810.09126}, Oct. 2018.

\bibitem{Cardoso2019}
A.~R. Cardoso and H.~Xu, ``Risk-{A}verse {S}tochastic {C}onvex {B}andit,'' in
  \emph{International Conference on Artificial Intelligence and Statistics},
  vol.~89, Apr. 2019, pp. 39--47.

\bibitem{W.Huang2017}
W.~Huang and W.~B. Haskell, ``Risk-{A}ware {Q}-learning for {M}arkov {D}ecision
  {P}rocesses,'' in \emph{2017 IEEE 56th Annual Conference on Decision and
  Control, CDC 2017}, vol. 2018-Janua.\hskip 1em plus 0.5em minus 0.4em\relax
  IEEE, Dec. 2018, pp. 4928--4933.

\bibitem{Jiang2017}
D.~R. Jiang and W.~B. Powell, ``Risk-{A}verse {A}pproximate {D}ynamic
  {P}rogramming with {Q}uantile-{B}ased {R}isk {M}easures,'' \emph{Mathematics
  of Operations Research}, vol.~43, no.~2, pp. 554--579, Nov. 2018.

\bibitem{Kalogerias2018b}
D.~S. Kalogerias and W.~B. Powell, ``Recursive {O}ptimization of {C}onvex
  {R}isk {M}easures: {M}ean-{S}emideviation {M}odels,'' \emph{arXiv preprint,
  arXiv:1804.00636}, Apr. 2018.

\bibitem{Tamar2017}
A.~Tamar, Y.~Chow, M.~Ghavamzadeh, and S.~Mannor, ``Sequential {D}ecision
  {M}aking with {C}oherent {R}isk,'' \emph{IEEE Transactions on Automatic
  Control}, vol.~62, no.~7, pp. 3323--3338, Jul. 2017.

\bibitem{Vitt2018}
C.~A. Vitt, D.~Dentcheva, and H.~Xiong, ``Risk-{A}verse {C}lassification,''
  \emph{Annals of Operations Research}, Aug. 2019.

\bibitem{Zhou2018}
L.~Zhou and P.~Tokekar, ``An {A}pproximation {A}lgorithm for {R}isk-averse
  {S}ubmodular {O}ptimization,'' \emph{arXiv preprint, arXiv:1807.09358}, Jul.
  2018.

\bibitem{Ruszczynski2010}
A.~Ruszczy{\'{n}}ski, ``Risk-{A}verse {D}ynamic {P}rogramming for {M}arkov
  {D}ecision {P}rocesses,'' \emph{Mathematical Programming}, vol. 125, no.~2,
  pp. 235--261, Oct. 2010.

\bibitem{SOPASAKIS2019281}
P.~Sopasakis, D.~Herceg, A.~Bemporad, and P.~Patrinos, ``Risk-{A}verse {M}odel
  {P}redictive {C}ontrol,'' \emph{Automatica}, vol. 100, pp. 281 -- 288, 2019.

\bibitem{kalogerias2020noisy}
D.~S. Kalogerias, ``{N}oisy {L}inear {C}onvergence of {S}tochastic {G}radient
  {D}escent for {CV}@{R} {S}tatistical {L}earning under {P}olyak-{L}ojasiewicz
  conditions,'' \emph{arXiv preprint arXiv:2012.07785}, 2020.

\bibitem{ShapiroLectures_2ND}
A.~Shapiro, D.~Dentcheva, and A.~Ruszczy{\'n}ski, \emph{Lectures on Stochastic
  Programming: Modeling and Theory}, 2nd~ed.\hskip 1em plus 0.5em minus
  0.4em\relax Society for Industrial and Applied Mathematics, 2014.

\bibitem{Rockafellar1997}
R.~T. Rockafellar and S.~Uryasev, ``Optimization of {C}onditional
  {V}alue-at-{R}isk,'' \emph{Journal of Risk}, vol.~2, pp. 21--41, 1997.

\bibitem{chapman2022cvarLQ}
M.~P. Chapman and L.~Lessard, ``{T}oward a {S}calable {U}pper {B}ound for a
  {CV}a{R}-{LQ} {P}roblem,'' \emph{IEEE Control Systems Letters}, vol.~6, pp.
  920--925, 2022.

\bibitem{zhou1996robust}
K.~Zhou, J.~Doyle, and K.~Glover, \emph{{R}obust and {O}ptimal
  {C}ontrol}.\hskip 1em plus 0.5em minus 0.4em\relax Prentice Hall, 1996.

\bibitem{tzortzis2016robust}
I.~{Tzortzis}, C.~D. {Charalambous}, T.~{Charalambous}, C.~K. {Kourtellaris},
  and C.~N. {Hadjicostis}, ``Robust {L}inear {Q}uadratic {R}egulator for
  {U}ncertain {S}ystems,'' in \emph{2016 IEEE 55th Conference on Decision and
  Control (CDC)}, 2016, pp. 1515--1520.

\bibitem{dean2020sample}
S.~Dean, H.~Mania, N.~Matni, B.~Recht, and S.~Tu, ``On the {S}ample
  {C}omplexity of the {L}inear {Q}uadratic {R}egulator,'' \emph{Foundations of
  Computational Mathematics}, vol.~20, no.~4, pp. 633--679, 2020.

\bibitem{glover1988state}
K.~Glover and J.~C. Doyle, ``State-{S}pace {F}ormulae for {A}ll {S}tabilizing
  {C}ontrollers that {S}atisfy an ${H}_{\infty}$-norm {B}ound and {R}elations
  to {R}isk {S}ensitivity,'' \emph{Systems \& control letters}, vol.~11, no.~3,
  pp. 167--172, 1988.

\bibitem{zhang2021derivative}
K.~Zhang, X.~Zhang, B.~Hu, and T.~Ba{\c{s}}ar, ``Derivative-{F}ree {P}olicy
  {O}ptimization for {R}isk-{S}ensitive and {R}obust {C}ontrol {D}esign:
  {I}mplicit {R}egularization and {S}ample {C}omplexity,'' \emph{arXiv preprint
  arXiv:2101.01041}, 2021.

\bibitem{goel2021regret}
G.~Goel and B.~Hassibi, ``Regret-optimal measurement-feedback control,'' in
  \emph{Learning for Dynamics and Control}.\hskip 1em plus 0.5em minus
  0.4em\relax PMLR, 2021, pp. 1270--1280.

\bibitem{tsiamis2020risk}
A.~Tsiamis, D.~S. Kalogerias, L.~F. Chamon, A.~Ribeiro, and G.~J. Pappas,
  ``Risk-{C}onstrained {L}inear-{Q}uadratic {R}egulators,'' in \emph{59th IEEE
  Conference on Decision and Control (CDC)}, 2020, pp. 3040--3047.

\bibitem{zhao2021infinite}
F.~Zhao, K.~You, and T.~Basar, ``{I}nfinite-horizon {R}isk-constrained {L}inear
  {Q}uadratic {R}egulator with {A}verage {C}ost,'' \emph{arXiv preprint
  arXiv:2103.15363}, 2021.

\bibitem{zhao2021global}
F.~Zhao, K.~You, and T.~Ba{\c{s}}ar, ``{G}lobal {C}onvergence of {P}olicy
  {G}radient {P}rimal-dual {M}ethods for {R}isk-constrained {LQR}s,''
  \emph{arXiv preprint arXiv:2104.04901}, 2021.

\bibitem{abeille2016lqg}
M.~Abeille, A.~Lazaric, X.~Brokmann \emph{et~al.}, ``{LQG} for {P}ortfolio
  {O}ptimization,'' \emph{Available at SSRN:
  https://ssrn.com/abstract=2863925}, 2016.

\bibitem{bertsekas2017dynamic}
D.~P. Bertsekas, \emph{Dynamic Programming and Optimal Control}, 4th~ed.\hskip
  1em plus 0.5em minus 0.4em\relax Athena Scientific, 2017, vol.~1.

\bibitem{anderson2005optimal}
B.~Anderson and J.~Moore, \emph{Optimal Filtering}.\hskip 1em plus 0.5em minus
  0.4em\relax Dover Publications, 2005.

\bibitem{ruszczynski2011nonlinear}
A.~Ruszczynski, \emph{Nonlinear Optimization}.\hskip 1em plus 0.5em minus
  0.4em\relax Princeton university press, 2011.

\bibitem{lewis1981generalized}
F.~Lewis, ``A {G}eneralized {I}nverse {S}olution to the {D}iscrete-{T}ime
  {S}ingular {R}iccati {E}quation,'' \emph{IEEE Transactions on Automatic
  Control}, vol.~26, no.~2, pp. 395--398, 1981.

\bibitem{gattami2009generalized}
A.~Gattami, ``Generalized {L}inear {Q}uadratic {C}ontrol,'' \emph{IEEE
  Transactions on Automatic Control}, vol.~55, no.~1, pp. 131--136, 2009.

\bibitem{bertsekas2012approximate}
D.~P. Bertsekas, \emph{Dynamic Programming and Optimal Control}, 4th~ed.\hskip
  1em plus 0.5em minus 0.4em\relax Athena Scientific, 2012, vol. 2: Approximate
  Dynamic Programming.

\end{thebibliography}
\appendix
\section*{Proof of Theorem~\ref{thm:OPTIMAL}}
To prove part 1), let $\mu_2> \mu_1\ge 0$. From the definition of the Lagrangian and optimality of the controller $u^*(\mu)$, we obtain the inequalities
\begin{align*}
    J(u^*(\mu_1))+\mu_1 J_R(u^*(\mu_1))&\le  J(u^*(\mu_2))+\mu_1 J_R(u^*(\mu_2))\\
    J(u^*(\mu_1))+\mu_2 J_R(u^*(\mu_1))&\ge J(u^*(\mu_2))+\mu_2 J_R(u^*(\mu_2))  .
\end{align*}
By subtracting, we get
\[
(\mu_2-\mu_1)\set{J_R(u^*(\mu_1))-J_R(u^*(\mu_2))}\ge 0,
\]
which shows that $J_R(u^*(\mu_1))\ge J_R(u^*(\mu_2))$. The proof of
$J(u^*(\mu_1))\le J(u^*(\mu_2))$ is similar.

To prove part 2), we first show that, whenever $\mu^*<\infty$, $\mu^*\Neg{0.6}(J_R(u^*\Neg{0.6}(\mu^*))\Neg{1}-\Neg{1}\bar{\epsilon})\Neg{2.2}=\Neg{2}0$, i.e., complementary slackness holds. 
We have two cases: either $\mu^*=0$, where complementary slackness is satisfied trivially; or $\mu^*>0$, $J_R(u^*(\mu^*))\le \bar{\epsilon}$. Therefore, it will be sufficient to show that in the latter case we can only have $J_R(u^*(\mu^*))= \bar{\epsilon}$.
Since $\mu^*>0$, it is true that $J_R(u^*(0))>\bar{\epsilon}$. Now, assume that $J_R(u^*(\mu^*))< \bar{\epsilon}$. Then by the assumption of continuity of $J_R(u^*(\mu))$, there exists a $0<\bar{\mu}<\mu^*$ such that $J_R(u^*(\bar{\mu}))=\bar{\epsilon}$, contradicting the definition of $\mu^*$. Hence, we can only have $J_R(u^*(\mu^*))= \bar{\epsilon}$, which shows that complementary slackness is satisfied. 

Now, complementary slackness, along with the trivial fact that $J_R(u^*(\mu^*))\le \bar{\epsilon}$ imply that the policy-multiplier pair $(u^*(\mu^*),\mu^*) \in \mathcal{U}_0 \times \mathbb{R}_+$ satisfies the sufficient conditions for optimality provided by Theorem~\ref{thm:KKT}. Enough said. 

To prove the last claim of part 2), suppose that \eqref{eq:LQ_constrained_reformulated} 
 satisfies
Slater's condition. For every $\mu\ge0$, we have
\begin{equation}
\begin{aligned}
    D(\mu) &\le J(u^\dagger) + \mu (J_R(u^\dagger)-\bar{\epsilon}) \nonumber\\
\implies D(\mu) - \mu (J_R(u^\dagger)-\bar{\epsilon}) &\le J(u^\dagger)<\infty. \nonumber
\end{aligned}
\end{equation}
Next, suppose that, for every $\mu\ge0$, $J_R(u^*(\mu))-\bar{\epsilon}\ge0$. Because $J(u^*(\cdot))$ is increasing on $\mathbb{R}_+$, it must be true that
\begin{equation}
\begin{aligned}
    J(u^\dagger) \Neg{2}&\ge\Neg{1} \sup_{\mu\ge0} D(\mu) \Neg{1}-\Neg{1} \mu (J_R(u^\dagger)\Neg{1}-\Neg{1}\bar{\epsilon}) \nonumber\\
    & =\Neg{1} \sup_{\mu\ge0} J(u^*(\mu)) \Neg{1}+\Neg{1} \mu (J_R(u^*(\mu))\Neg{1}-\Neg{1}\bar{\epsilon}) \Neg{1}-\Neg{1} \mu (J_R(u^\dagger)\Neg{1}-\Neg{1}\bar{\epsilon}) \nonumber\\
    & = \Neg{1}\infty,
\end{aligned}
\end{equation}
which contradicts the fact that $J(u^\dagger)<\infty$. Therefore, there must exist $\mu^\dagger \ge0$, such that $J_R(u^*(\mu^\dagger))-\bar{\epsilon}<0$. But $J_R(u^*(\cdot))$ is decreasing on $\mathbb{R}_+$ and, consequently, it must be the case that $\mu^*\in[0,\mu^\dagger)$. The proof is now complete.
\hfill $\qedsymbol$

\section*{Proof of Theorem~\ref{thm:partially_observed}}
The quadratic and linear penalties $Q_{\mu,t}$, $m_{3,t}$, the errors $\delta_t,\epsilon_t$, and all parameters $V_t,\xi_t,d_t,K_t,l_t,P_t,\zeta_t,c_t$ are a function of the stochastic dynamics of the system. To see why this holds, 
define the stochastic part of the system as:
\[
x^s_k\triangleq x_k-\sum_{i=0}^{k}BA^{i}u_{k-i}=Ax^s_{k-1}+w_k,\,y^s_k\triangleq Cx^s_k+v_k,
\]
with $x^s_0=x_0$.
We can define the deterministic part of the system as:
\[
x^d_t\triangleq \sum_{i=0}^{k}BA^{i}u_{k-i}=x_t-x^s_t.
\]
Since the deterministic part $x^d_t\in\F_{t-1}$ is measurable with respect to the current information:
\[
\E(x^d_t|\F_{t-1})=x^d_t.
\]
As a result,
\[
\delta_t=x^s_t-\E(x^s_t|\F_{t-1}),\,e_t=\E(x^s_t|\F_{t})-\E(x^s_t|\F_{t-1}).
\]
Let now $u^1,u^2\in \mathcal{U}_0$ be two arbitrary policies. Denote by $x^{s,1}_t$, $x^{s,2}_t$ the respective stochastic component of the state. Then for any sample of the probability space (everywhere):
\[
x^{s,1}_t=x^{s,2}_t.
\]
In other words, the stochastic states are point-wise independent of the inputs. The same holds for the quadratic and linear penalties $Q_{\mu,t}$, $m_{3,t}$, the errors $\delta_t,\epsilon_t$, and all parameters $V_t,\xi_t,d_t,K_t,l_t,P_t,\zeta_t,c_t$.
Then, this recursively implies that pointwise everywhere:
\[
\frac{\partial V_t}{\partial u_t}=0,\, \frac{\partial \xi_t}{\partial u_t}=0,\, \frac{\partial d_t}{\partial u_t}=0.
\]
Meanwhile, since all moments of $w_t$ exist and $R\succ 0$ is strictly positive definite, all moments of $V_t,\xi_t,d_t,K_t,l_t,P_t,\zeta_t,c_t$ also exist (follows from Hölder's inequality). In particular, all moments of $K_t$, $l_t$ exist.

By using dynamic programming and assuming (temporarily) that involved measurability issues are resolved~\cite{bertsekas2012approximate},  we have, for every $k\le N-1$, the recursive optimality condition (i.e., the Bellman equation)
\begin{align*}
 \LAG^*_k(z_{k},\mu)=\inf_{u_{k}} g_{k+1}(z_k,u_k,\mu)+\E\set{ \LAG^*_{k+1}(z_{k+1},\mu)|\F_{k}}. 
 \end{align*}
The base case is obvious since $\LAG^*_N(z_N,\mu)=0$. Assume it is true for $k=t+1$, we will show that the same holds for $k=t$. Writing $\hat{x}_{t+1|t+1}=\hat{x}_{t+1}+e_{t+1}$, we have:
\begin{align*}
   & \E\set{ \LAG^*_{t+1}(z_{t+1},\mu)|\F_{t}}\\
    &=\E\set{\hat{x}'_{t+1|t+1}P_{t+1}\hat{x}'_{t+1|t+1}+2\zeta'_{t+1}\hat{x}_{t+1|t+1}+c_{t+1}|\F_{t}}\\
    &=\hat{x}'_{t+1}\E(P_{t+1}|\F_t)\hat{x}_{t+1}+2\E(e'_{t+1}P_{t+1}+\zeta'_{t+1}|\F_{t})\hat{x}_{t+1} \\
    &+\E(c_{t+1}|\F_t)+\E(e'_{t+1}P_{t+1}e_{t+1}|\F_{t})+2\E(\zeta'_{t+1}e_{t+1}|\F_{t}).
\end{align*}
As a result, we obtain the following quadratic form
\begin{equation}\label{OUT_eq:quadratic_cost_to_go}
\begin{aligned}
 &g_{t}(z_t,u_t,\mu)+\E\set{ \LAG^*_{t+1}(z_{t+1},\mu)|\F_{t}}\\
 &=\hat{x}'_{t+1}V_t\hat{x}_{t+1}+u'_tRu_t+2\xi'_t \hat{x}_{t+1}+d_t.
\end{aligned}
\end{equation}
Based on the above we can also verify that all measurability issues are now resolved in a recursive way, retrospectively. Recall that $\hat{x}_{t+1}=A\hat{x}_{t|t}+Bu_t+\bar{w}$. 
Since all $V_t,\xi_t,d_t$ are independent of $u_t$ pointwise, the unique stationary point of the above convex quadratic form is given by:
\[
u^*_t=K_t\hat{x}_{t|t}+l_t.
\]
Plugging the optimal input~\eqref{eq:LQ_Optimal_Input} into~\eqref{OUT_eq:quadratic_cost_to_go} gives the optimal cost-to-go~\eqref{eq:optimal_cost_to_go}. In detail:
\begin{align*}
  &\set{A\hat{x}_{t|t}+Bu^*_t+\bar{w}}'V_t \set{A\hat{x}_{t|t}+Bu^*_t+\bar{w}}\\
  &+(u^*_t)'Ru^*_t+2\xi'_t\set{A\hat{x}_{t|t}+Bu^*_t+\bar{w}}+d_t\\
  &=\hat{x}'_{t|t} P_t\hat{x}_{t|t}+2\set{Bl_t+\bar{w}}'V_t (A+BK_t)\hat{x}_{t|t}\\
  &+\set{Bl_t+\bar{w}}'V_t \set{Bl_t+\bar{w}}+2l'_tRK_t\hat{x}_{t|t}+l'_tRl_t\\
  &+2\xi'_t(A+BK_t)\hat{x}_{t|t}+2\xi'_t(Bl_t+\bar{w})+d_t\\
  &\stackrel{i)}{=}\hat{x}'_{t|t} P_t\hat{x}_{t|t}+2\bar{w}'V_t(A+BK_t)\hat{x}_{t|t}\\
  &+2\xi'_t(A+BK_t)\hat{x}_{t|t}+\set{Bl_t+\bar{w}}'V_t \set{Bl_t+\bar{w}}\\
  &+l'_tRl_t+2\xi'_t(Bl_t+\bar{w})+d_t\\
  &\stackrel{ii)}{=}\hat{x}'_{t|t} P_t\hat{x}_{t|t}+2\bar{w}'V_t(A+BK_t)\hat{x}_{t|t}+2\xi'_t(A+BK_t)\hat{x}_{t|t}\\
  &+\bar{w}'V_t \bar{w}-l'_t(B'V_tB+R)l_t+2\xi_t'\bar{w}+d_t
\end{align*}
where the cancellations in $i)$ follow from the identity:
\[
(B'V_tB+R)K_t=-B'V_tA,
\]
and $ii)$ follows from:
\[
2\xi'Bl_t+2\bar{w}'V_tBl_t=-2l'_t(B'V_tB+R)l_t.
\]
Recursively, we can verify that the optimal control has finite fourth moments. In fact, all higher-order moments of $u^*_t$ exist. This follows from the fact that all moments of $K_t$, $l_t, w_t$ exist.
\hfill $\qedsymbol$

\section*{Proof of Proposition~\ref{prop:LQ_Risk_Evaluation}}
We will only sketch the proof of continuity. The proof of the recursive expressions is omitted since it is similar to the proof of Theorem~\ref{thm:partially_observed}. Let $\mu_s\ge 0$, $s=1,2,\dots$ be a sequence such that $\mu_s\rightarrow \mu$ and let $\snorm{\mu}_{\infty}\triangleq \sup_{s\ge 1} \mu_s$. We will use Dominated Convergence Theorem (DCT) for conditional expectation to prove that $J_R(u^*(\mu_s))$ converges to $J_R(u^*(\mu))$. To emphasize the dependence on $\mu_s$, we will use the notation $V_t(\mu_s),\dots,c_t(\mu_s)$ and $\Theta_t(\mu_s),\dots,g_{t}(\mu_s)$ for the quantities appearing in Theorem~\ref{thm:partially_observed} and in the statement.

We will show that the terms $V_{N-2}(\mu_s),\, H_{N-2}(\mu_s)$ all converge almost surely to their respective limits. The proof for the remaining terms is similar. The idea is to show that $K_{N-1}$ is dominated by a function which is independent of the index $s$. Note that $\snorm{V_{N-1}(\mu_s)}_2\le \snorm{\mu}_{\infty}\snorm{W_{N-1}}_2\snorm{Q}^2_2+\snorm{Q}_2$. Since $R$ is invertible, we obtain:
\[
\snorm{K_{N-1}(\mu_s)}_2\le O(\snorm{W_{N-1}}_2+1).
\]
We interpret the notation $\alpha=O(\beta)$ as follows: there is a deterministic constant $\mathcal{C}=\mathcal{C}(\snorm{\mu}_{\infty})$  such that $\alpha\le \mathcal{C}\beta$ almost surely. As a result, we also obtain:
\[
\snorm{P_{N-1}(\mu_s)}_2,\snorm{\Theta_{N-1}(\mu_s)}_2 \le O(\snorm{W_{N-1}}_2^3+1).
\]
Note that since all moments of $\delta_{N}$ exist (follows from Assumption~\ref{ass:moments_partially_observed}), the term $\snorm{W_{N-1}}^3$ has finite expectation. Meanwhile, almost surely $P_{N-1}(\mu_s)\rightarrow P_{N-1}(\mu)$, $\Theta_{N-1}(\mu_s)\rightarrow \Theta_{N-1}(\mu)$. Hence, by the conditional DCT, we also have:
\begin{align*}
\E(P_{N-1}(\mu_s)|\F_{N-2})&\rightarrow \E(P_{N-1}(\mu)|\F_{N-2}),\text{ a.s. }\\
\E(\Theta_{N-1}(\mu_s)|\F_{N-2})&\rightarrow \E(\Theta_{N-1}(\mu)|\F_{N-2}),\text{ a.s. }
\end{align*}
This proves that $V_{N-2}(\mu_s),\, H_{N-2}(\mu_s)$ converge to $V_{N-2}(\mu),\, H_{N-2}(\mu)$ almost surely. Meanwhile, they remain dominated by:
\begin{align*}
&\snorm{V_{N-2}(\mu_s)}_2,\snorm{H_{N-2}(\mu_s)}_2 \le\\ &O(\E(\snorm{W_{N-1}}_2^3|\F_{N-2})+\snorm{W_{N-2}}_2+1).
\end{align*}

Proceeding similarly, by successive use of DCT and the existence of all moments of $\delta_t$, we can show that all matrices $V_t(\mu_s),\,P_t(\mu_s),\,H_t(\mu_s),\,\Theta_t(\mu_s)$ and their (conditional) means converge to the respective limits almost surely. We omit the details to avoid repetition. The same will hold for terms $\eta_t(\mu_s)$ and $\gamma_0(\mu_s)$ and their (conditional) means. \hfill $\qed$

\section*{Proof of Theorem~\ref{thm:gaussian_stability}}
For convenience, we drop the index $\mu$ from $Q_{\mu,t}$. 
Note that from standard Kalman filter theory~\cite{anderson2005optimal}[Ch. 4.4], the covariance $W_t$ converges exponentially fast to $W_{\infty}$ as $t_0$ goes to $-\infty$. Moreover, since $W_{t_0}=0$, the sequence $W_t$ is increasing with respect to the positive semi-definite cone~\cite{anderson2005optimal}[Ch. 4.4].
Since $Q_{\infty}\succeq Q$, we also obtain that $(A,Q^{1/2}_{\infty})$ is detectable. Hence, the stabilizing solution $V$ is well-defined and $A+BK$ is stable. 

We will only prove~\eqref{eq:lqg_finite_time_error}. The remaining results follow using similar arguments.
First we state a lemma that generalizes Problem 4.5, Section 4 in~\cite{anderson2005optimal}.
\begin{lemma}\label{lem:app_RDE}
Let $V_t,\,\bar{V}_t$ satisfy the Riccati Difference Equations
\begin{align*}
    V_{t-1}&=(A+BK_t)'V_{t}(A+BK_t)+Q_{t-1}+K'_tRK_t\\
    \bar{V}_{t-1}&=(A+BL_t)'\bar{V}_{t}(A+BL_t)+\bar{Q}_{t-1}+L'_tRL_t,
\end{align*}
with gains
\begin{align*}
    K_t&=-(B'V_{t}B+R)^{-1}B'V_{t}A\\
    L_t&=-(B'\bar{V}_{t}B+R)^{-1}B'\bar{V}_{t}A,
\end{align*}
Then, their difference satisfies the identity:
\[
V_{t-1}-\bar{V}_{t-1}=(A+BL_t)'(V_{t}-\bar{V}_{t})(A+BK_{t})+Q_{t-1}-\bar{Q}_{t-1}
\]
\end{lemma}
\begin{proof}
Omitted; similar to the proof of Problem~4.5 in~\cite{anderson2005optimal}.
\end{proof}

For $N-1\ge k_1>k_2$, define the products:
\begin{equation}\label{eq:psi}
\Psi_{k_1:k_2}(N)\triangleq (A+BK_{k_1})\times \cdots \times (A+BK_{k_2+1}),
\end{equation}
with $\Psi_{k_1:k_1}(N)=I$.
Applying Lemma~\ref{lem:app_RDE} to $V_{t}$ and $V$, we get:
\begin{equation*}
    V_{t}-V=\bar{A}'(V_{t+1}-V)(A+BK_{t+1})+Q_{t}-Q_{\infty}.
\end{equation*}
Repeating multiple times, we obtain a fundamental identity:
\begin{equation}\label{eq:fundamental_difference_RDE}
\begin{aligned}
    V_{t}-V&=(\bar{A}')^{N-t-1}(Q_{N-1}-V)\Psi_{N-1:t}(N)\\
    &+\sum_{k=0}^{N-t-2} (\bar{A}')^{k} (Q_{t+k}-Q_{\infty})\Psi_{t+k:t}(N).
    \end{aligned}
\end{equation}
This almost gives us the result. What remains is to show that $\Psi_{t+k:t}(N)$ are uniformly bounded over all $k,\,N,\,t$. This follows from the following lemma.
\begin{lemma}\label{lem:uniform_boundedness}
Let $Q$ be positive definite. For any $k_1\ge k_2\ge t_0$, $N\ge t_0$ such that $N-1\ge k_1$:
\begin{equation*}
    \snorm{\Psi_{k_1:k_2}(N)}_2 \le \sqrt{\sigma^{-1}_{\min}(Q)\snorm{V}_2}.
\end{equation*}
\end{lemma}
\begin{proof} It is sufficient to show that:
\[
V\succeq \Psi_{k_1:k_2}(N)'Q\Psi_{k_1:k_2}(N).
\]
Then, we can get the result from the bound
\[
\snorm{V}_2\ge \sigma_{\min}(Q) \snorm{\Psi_{k_1:k_2}(N)}^2_2.
\]

The proof will proceed in two steps. First, we show that for any $k,\,N$ such that $k\ge t_0, N-1\ge k$ the stabilizing solution $V$ overbounds $V_k$. Second, we use this property to upper-bound the products $\Psi_{k_1:k_2}(N)$.

\textbf{Step 1.} We show that $V_k\preceq V$ via induction. Recall that the sequence $W_t$ is increasing, which implies that $Q_{t}\preceq Q_{\infty}$, for any $t\ge t_0$.
For the base case, we have $V_{N-1}=Q_{N-1}\preceq Q_{\infty}\preceq V$. Assume that $V_{t} \preceq V$. Define the operator:
\[
g(F)\triangleq A'FA+Q_{\infty}-A'FB(B'FB+R)^{-1}B'FA.
\] Then, since $Q_{t-1}\preceq Q_{\infty}$, we have:
\[
V_{t-1}\preceq g(V_{t})\stackrel{i)}{\preceq} g(V)=V.
\]
The second inequality i) follows from the fact (e.g. see pages 79-80 Ch. 4.4 of~\cite{anderson2005optimal}) that $g(\cdot)$ preserves positive semi-definite order, i.e. if $F_1\succeq F_2$, then $g(F_1)\succeq g(F_2)$. This completes the proof of the first step.

\textbf{Step 2.} We will show that:
\[
V_{k_2}\succeq \Psi_{k_1:k_2}(N)'Q\Psi_{k_1:k_2}(N).
\]
Then, using the result of step 1 proves the desired inequality.
From the definition of the Riccati Difference Equation:
\begin{align*}
V_{k_2}&\succeq (A+BK_{k_2+1})'V_{k_2+1}(A+BK_{k_2+1})\\
&\succeq \cdots \succeq \Psi_{k_1:k_2}(N)'V_{k_1}\Psi_{k_1:k_2}(N)'.
\end{align*}
To complete the proof, note that $V_{k_1}\succeq Q_{k_1}\succeq Q$ for any $k_1\le N-1$.
\end{proof}
Now choose:
\begin{align*}
\C_1&= \snorm{V}^{3/2}_2 \sqrt{\sigma^{-1}_{\min}(Q)}\\
\tilde{\C}_2&= \sqrt{\sigma^{-1}_{\min}(Q)\snorm{V}_2} \sum^{\infty}_{k=0} \snorm{\bar{A}^k}_2
\end{align*}
Based on the above lemma, we obtain:
\begin{align*}
    \snorm{V_t-V}_2 \le \C_1\snorm{\bar{A}^{N-t-1}}_2 +\tilde{\C}_2 \sup_{k\ge t}\snorm{Q_k-Q_{\infty}}_2.
\end{align*}
Since the sequence $W_t$ is increasing, we can replace the supremum by
\[
\sup_{k\ge t}\snorm{Q_k-Q_{\infty}}_2=\snorm{Q_t-Q_{\infty}}_2.
\]
Finally, since:
\[
Q_t-Q_{\infty}=\mu Q(W_t-W_{\infty})Q,
\]
if suffices to select $\C_2=\mu \tilde{\C}_2\snorm{Q}^2$. \hfill $\qed$

Note that if $Q$ is singular, then the result of Lemma~\ref{lem:uniform_boundedness} no longer applies. Instead we could bound the products $V^{1/2}\Psi(k_1:k_2)(N)$. A way to do this is to use $N-k_1$ large enough in the proof of step 2, so that $V_{k_1}$ has the same range space as $V$. To avoid technicalities, we defer the proof for future work. 
\end{document}